\documentclass[11pt]{amsproc}
\usepackage{amssymb, amsthm, amscd, epsfig, graphicx, 
enumerate, stmaryrd, amsmath, graphs, xcolor}

\newtheorem{thm}{Theorem}[section]

\newtheorem{lem}[thm]{Lemma}
\newtheorem{prop}[thm]{Proposition}

\theoremstyle{definition}
\newtheorem{defn}[thm]{Definition}

\theoremstyle{remark}

\numberwithin{equation}{section}

\newcommand{\al}{\alpha}

\newcommand{\ga}{\gamma}

\newcommand{\de}{\delta}
\newcommand{\ep}{\varepsilon}

\renewcommand{\phi}{\varphi}

\renewcommand{\rho}{\varrho}

\newcommand{\x}{\times}

\newcommand{\Z}{\mathbb Z}
\newcommand{\N}{\mathbb N}

\newcommand{\R}{\mathbb R}

\newcommand{\RP}{{\mathbb R}{P}}

\newcommand{\del}{\partial}

\newcommand{\co}{\colon}

\newcommand{\cl}{\mathrm{cl}\thinspace}
\newcommand{\interior}{\mathrm{int} \thinspace}

\makeindex
\begin{document}
\mathsurround=1pt 
\title{Decomposition space theory}
\author{Boldizs\'ar Kalm\'ar}
%


%



%

\maketitle
\setcounter{tocdepth}{2}

%
%
%
%
%


\large

{\Large{\section{Introduction}}}
\bigskip\large

In these notes we give a brief introduction to decomposition theory
and we summarize some classical and well-known results. 
The main question is that if a partitioning of a topological space (in other words a \emph{decomposition}) is given, then what is the topology of the quotient space.
The main result is that an
\emph{upper semi-continuous} decomposition  yields a homeomorphic decomposition space if
the decomposition is \emph{shrinkable} (i.e.\ there exist self-homeomorphisms of the space which shrink the 
partitions into arbitrarily small sets in a controllable way). 
This is called \emph{Bing shrinkability criterion} and it was introduced in \cite{Bi52, Bi57}. 
It is applied in major $4$-dimensional results:
in the disk embedding theorem and in the proof of the $4$-dimensional topological 
Poincar\'e conjecture \cite{Fr82, FQ90, BKKPR}.
It is extensively applied in constructing  approximations of manifold embeddings in dimension $\geq 5$, see
\cite{AC79} and Edwards's cell-like approximation theorem \cite{Ed78}.
If a decomposition is shrinkable, then a decomposition element has to be \emph{cell-like} and \emph{cellular}.
Also the quotient  map is approximable by homeomorphisms.
 A {cell-like map} is a map where the point preimages are similar to points while
 a {cellular map} is a map where the point preimages can be approximated by balls.
There is an essential difference between the two types of maps: 
ball approximations  always give cell-like sets but in a smooth manifold for a cell-like set $C$ the complement  has to 
be simply connected in a nbhd of $C$ in order to be cellular.
Finding conditions for a decomposition to be shrinkable 
is one of the main goal of the theory.
For example, cell-like decompositions are shrinkable if
the non-singleton decomposition elements
have codimension $\geq 3$, that is any maps of disks can be made disjoint from them \cite{Ed16}. 
In many constructions 
 Cantor sets (a set of uncountably many points
that cutting out from the real line we are left with a manifold)
 arise as limits of sequences of sets defining the decomposition.
 The interesting fact is that a limit Cantor set can be non-standard and it can have properties 
 very different from the usual middle-third Cantor set in $[0,1]$.
 An example for such a non-standard Cantor set 
 is given by Antoine's necklace but many other explicit constructions are studied in the subsequent sections.
 The present notes will cover the following:
 upper semi-continuous decompositions,
 defining sequences, cellular and cell-like sets, 
 examples like Whitehead continuum, Antoine's necklace and Bing decomposition, 
 shrinkability criterion and near-homeomorphism, approximating by homeomorphisms 
 and shrinking countable upper semi-continuous decompositions.
 We prove for example 
 that every cell-like subset in a $2$-dimensional manifold is cellular,
 that Antoine's necklace is a wild Cantor set,
 that in a complete metric space a usc decomposition is shrinkable if and only if the decomposition map is a near-homeomorphism and
 that every manifold has collared boundary.

\large

\bigskip{\Large{\section{Decompositions}}}
\bigskip\large

A neighborhood (nbhd for short) of a subset $A$ of a topological space $X$ is an open subset of $X$ which contains $A$.

\begin{defn}
Let  $X$ be a topological space.
A set $\mathcal D \subset \mathcal P (X)$ is a \emph{decomposition} of $X$ if the elements of $\mathcal D$ are pairwise disjoint and
$\bigcup \mathcal D = X$.
An element of $\mathcal D$ which consists of one single point is called a \emph{singleton}.
A non-singleton decomposition element is called \emph{non-degenerate}.
The elements of $\mathcal D$ are the \emph{decomposition elements}.
The set of non-degenerate elements is denoted by 
$\mathcal H_{\mathcal D}$.
\end{defn}

If $f \co X \to Y$ is an arbitrary (not necessarily continuous) map  between the topological spaces $X$ and $Y$, then 
the set $$\{ f^{-1}(y) : y \in Y \}$$ is a decomposition of $X$.
A decomposition defines an equivalence relation on $X$ as usual, i.e.\ $a, b \in X$ are equivalent iff $a$ and $b$ are in the same element of $\mathcal D$.

\begin{defn}
If $\mathcal D$ is a decomposition of $X$, then the \emph{decomposition space}
$X_\mathcal D$ is the space  $\mathcal D$ with the following topology:
the subset $U \subset \mathcal D$ is open exactly if
$\pi^{-1} (U)$ is open. Here $\pi \co X \to \mathcal D$ is the \emph{decomposition map} which maps each $x \in X$ into its 
equivalence class. 
\end{defn}
 
In other words $X_\mathcal D$ is the quotient space with the quotient topology and $$\pi \co X \to X_\mathcal D$$ is just the quotient map.
Recall that by well-known statements $X_\mathcal D$ is compact, connected and path-connected if $X$ is compact, connected and path-connected, respectively. Obviously $\pi$ is continuous.

\begin{prop}
The decomposition space is a $T_1$ space if the decomposition elements are closed.
\end{prop}
\begin{proof}
We have to show that  the points in the space
$X_{\mathcal D}$ are closed. If $U$ is a point complement in  $X_{\mathcal D}$, then $\pi^{-1}(U)$ is the complement of a decomposition 
element, which is open so $U$ is also open.
\end{proof}

We would like to construct and study such decompositions which 
have especially nice properties concerning the behavior of the sequences 
of decomposition elements.

\begin{defn}
Let $f \co X \to \R$ be a function. It is \emph{upper semi-continuous} (resp.\ \emph{lower semi-continuous}) if for every $x \in X$ and $\ep > 0$ there is a nbhd $V_x$ such that 
$f(V_x) \subset (-\infty, f(x) + \ep)$ (resp.\ $f(V_x) \subset (f(x) - \ep, \infty))$.
\end{defn}

For us, upper semi-continuous functions will be important. They are such functions, where a sequence $f ( x_n)$ 
can have only smaller or equal values than $f(x) + \ep_n$ as $x_n \rightarrow x$, where $\ep_n \geq 0$ and $\ep_n \rightarrow 0$.
Let $f \co \R \to \R$ be an upper semi-continuous, positive  function and consider the following decomposition of $\R^2$. Take the vertical segments of the form 
\begin{equation}\label{usc_example}
A_x = \{ (x, y) : y \in [0, f(x) ]\}
\end{equation}
 for each $x \in \R$.
Together with the points in $\R^2$ which are not in these segments (these points are the so-called singletons) this gives a decomposition of $\R^2$. 
This has an interesting property: let $y \in [0, f(x) ]$ for some $x \in \R$ and let $(x_n) \in \R$ be a sequence (which is not necessarily convergent). 
If every nbhd of the point $(x,y)$ intersects all but finitely many segments $A_{x_n}$, then the points $(u, v) \in \R^2$ 
each of whose nbhds
 intersects all but finitely many $A_{x_n}$ are in $A_x$ as well, see Figure~\ref{uscfunction}.
The set of the  points  $(u, v)$ is called the \emph{lower limit} of the sequence  $A_{x_n}$. 
In other words, if an $A_x$ intersects the lower limit of a sequence $A_{x_n}$, then 
all the lower limit is a subset of   $A_x$.
More generally we have the following.

\begin{figure}[h!]
\begin{center}
\epsfig{file=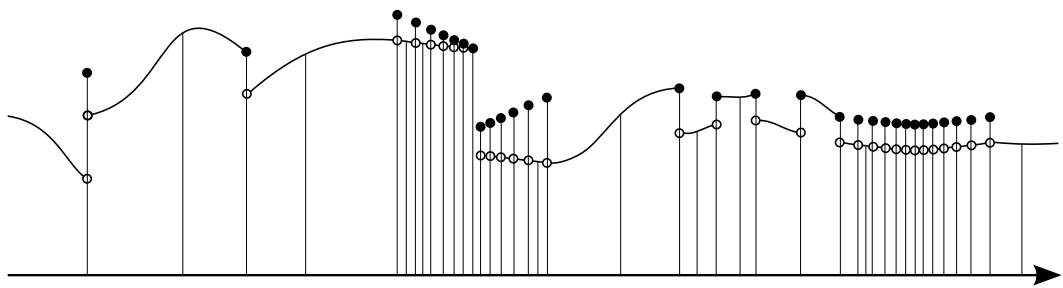, height=3.7cm}
\end{center} 
\caption{The graph of an upper semi-continuous function $f$ and some segments $A_x$.
If the segments $A_{x_n}$ ``converge'' to a segment $A_x$, then
$f(x_n)$ converges to a number $\leq f(x)$.}
\label{uscfunction}
\end{figure}

\begin{defn}
Let $A_n$ be a sequence of subsets of the space $X$. The \emph{lower limit} of $A_n$ is the set of the points $p \in X$ each of whose 
 nbhds intersects all but finitely many $A_{n}$. It is denoted by $\liminf A_n$.
 The \emph{upper limit} of $A_n$ is the set of the points $p \in X$ each of whose 
 nbhds intersects infinitely many $A_{n}$s. It is denoted by $\limsup A_n$.
\end{defn}

Note that $\liminf A_n \subset \limsup A_n$ is always true.
In the previous example the sets $A_{x_n}$ could approach the set $A_x$ only in a manner determined by the function $f$. This leads to  
the following general definition.

\begin{defn}
Let $\mathcal D$ be a decomposition of a space $X$ such that all elements of $\mathcal D$ are closed and compact and
they can converge to each other only in the following way:
if $A \in \mathcal D$, then for every nbhd $U$ of $A$ there is a nbhd $V$ of $A$  with the property $V \subset U$
such that if some  element $B \in \mathcal D$ intersects $V$, then $B \subset U$, i.e.\ the set $B$ is completely inside 
the nbhd $U$.
Then $\mathcal D$ is an \emph{upper semi-continuous decomposition} (\emph{usc} decomposition for short).
If all the decomposition elements are closed but not necessarily compact, then we say it is a 
\emph{closed upper semi-continuous decomposition}.
\end{defn}

For example, the decomposition defined in (\ref{usc_example}) is usc.

\begin{lem}\label{saturated}
Let $\mathcal D$ be a decomposition of the space $X$ such that each decomposition element is closed. The following are equivalent:
\begin{enumerate}[\rm (1)]
\item
$\mathcal D$ is a closed usc decomposition,
\item
for every 
$D \in \mathcal D$ and every  nbhd $U$  of $D$ 
there is a saturated nbhd $W \subset U$ of $D$, that is an open set  $W$ which is a union of decomposition elements,
\item
for each open subset $U \subset X$, the set $\cup\{ D  \in \mathcal D : D \subset U \}$ is open, 
\item
for each closed subset $F \subset X$, the set $\cup\{ D  \in \mathcal D : D \cap F \neq \emptyset  \}$ is closed, 
\item
the decomposition map $\pi \co X \to X_{\mathcal D}$ is a closed map.
\end{enumerate}
\end{lem}
\begin{proof}
Suppose  $\mathcal D$ is usc and $U$ is a nbhd of $D$.
Let $W$ be the union of all decomposition elements which are subsets of $U$. Then $D \subset W$ obviously and $W$ is open because
if $x \in W$, then $x \in D'$ for some decomposition element $D' \subset W$ and $D' \subset U$,
so by definition $D' \subset V \subset U$ for a nbhd $V$
but the nbhd $V$ of $x$ is in $W$ since 
all the decomposition elements intersecting $V$ have to be in $U$, which means they are in $W$ as well. This shows that 
(1) implies (2).
Suppose (2) holds. If $U$ is an open set, then 
for each decomposition element $D \subset U$
a saturated nbhd $W$ of $D$ is also in $U$ and also 
in $\cup\{ D  \in \mathcal D : D \subset U \}$.
This means that
 the set $\cup\{ D  \in \mathcal D : D \subset U \}$
 is a union of open sets, which proves (3).
 We have that (3) and (4) are equivalent because we can take the complement of a given closed set $F$ or an open set $U$.
We have that  (4) and (5) are equivalent: from (4) we can show (5) by
 taking an arbitrary closed set $F \subset X$,
 then $\cup\{ D  \in \mathcal D : D \cap F \neq \emptyset  \}$ is closed,
 its complement is a saturated open set whose $\pi$-image is open so
 $\pi(F)$ is closed. 
If we suppose (5), then for a closed set $F \subset X$ the set 
$$\pi^{-1} ( \pi(F)) = \cup\{ D  \in \mathcal D : D \cap F \neq \emptyset  \}$$ 
is closed so we get (4).
Finally  (3) implies (1):
 if $D \in \mathcal D$
 and $U$ is a nbhd of $D$, then 
 let $V$ be the open set 
 $\cup\{ D  \in \mathcal D : D \subset U \}$,
 this is a nbhd of $D$, it is in $U$ and
 if a $D' \in \mathcal D$ intersects $V$, then it is in $V$ and hence also in $U$.
 \end{proof}

There is also the notion of \emph{lower semi-continuous decomposition}: 
a decomposition $\mathcal D$ of a metric space is  lower semi-continuous  if
for every element $A \in \mathcal D$ and for every $\ep > 0$
there is a nbhd $V$ of $A$ such that if 
some decomposition element $B$ intersects $V$, then 
$A$ is in the $\ep$-nbhd  of $B$.
A decomposition of a metric space is \emph{continuous} if it is upper and lower semi-continuous, see Figure~\ref{lscdecomp}.
We will not study decompositions which are only lower semi-continuous.

\begin{figure}[h!]
\begin{center}
\epsfig{file=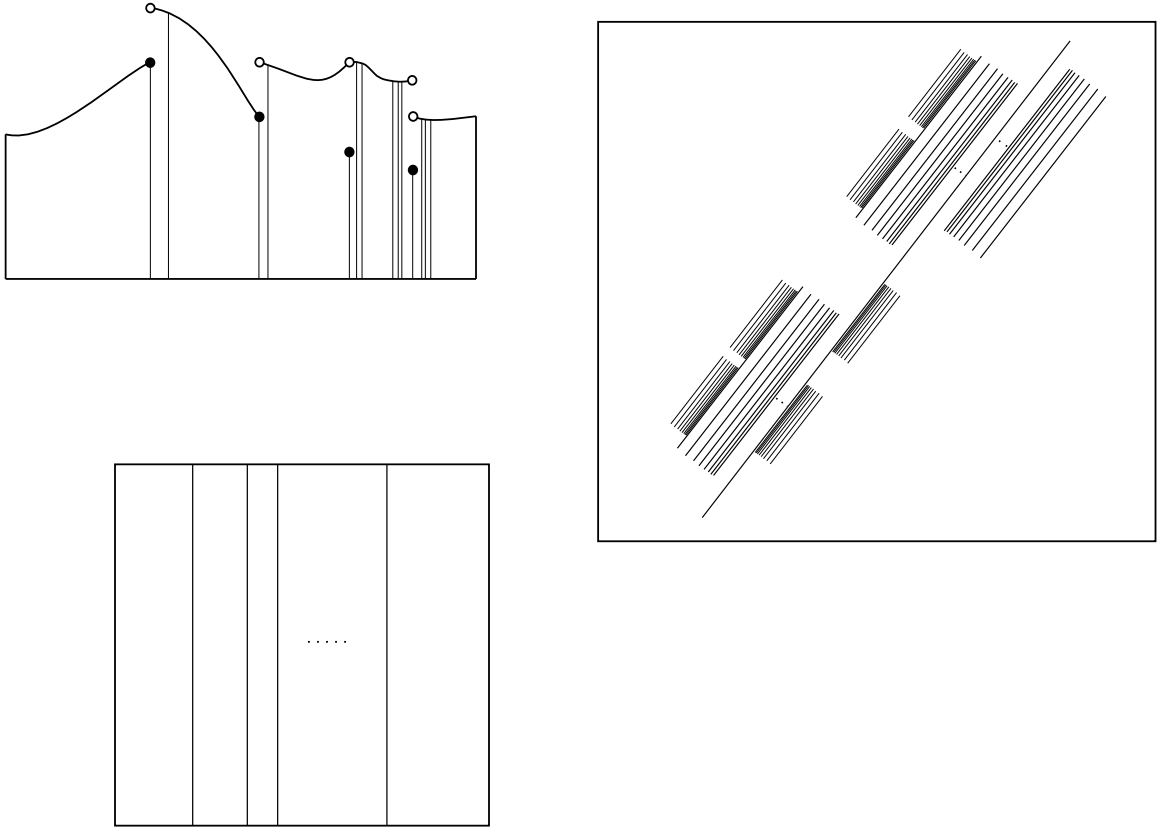, height=11cm}
\end{center} 
\caption{A lower semi-continuous, an upper semi-continuous and a continuous 
decomposition.
In each of the cases the non-degenerate decomposition elements are line segments, which
 converge to other line segments. The dots indicate convergence.
Only the non-singleton decomposition elements are sketched.
The lower semi-continuous decomposition
consists of decomposing 
the area under the graph of a lower semi-continuous function
into
vertical line segments, there are no singletons among the decomposition elements and
the decomposed space itself is not closed.
The upper semi-continuous and continuous decompositions
are decompositions of the rectangle. Only the upper semi-continuous decomposition
has singletons.}
\label{lscdecomp}
\end{figure}

\begin{thm}
Let $X$ be a $T_3$ space and  $\mathcal D$ is a closed usc decomposition.
If $A_n \in \mathcal D$ is a sequence of decomposition elements and $A \in \mathcal D$ are such that
$A \cap \liminf A_n \neq \emptyset$, then $\limsup A_n \subset A$.
\end{thm}
\begin{proof}
Suppose there is a point $x \in A$ such that $x \in \liminf A_n$ as well.
By contradiction suppose that $\limsup A_n \nsubseteq A$, this means that a point $y \in \limsup A_n$ is such that
$y \notin A$. Since $y \in D$ for a decomposition element, we get $D \neq A$ so $D$ is disjoint from the decomposition element $A$.
The space $X$ is $T_3$, the sets $D$ and $\{x\}$ are closed
so there is a nbhd $U$ of $D$ and a nbhd $V$ of $x$ which are
disjoint from each other.
We also have a nbhd $W \subset U$ of $D$ which is a union of decomposition elements by Lemma~\ref{saturated}.
Since $x \in  \liminf A_n$, we have that for an integer $k$ the sets $A_k, A_{k+1}, \ldots$ intersect $V$.
The nbhd $W$ is saturated, this implies that a decomposition element does not intersect both of $W$ and $V$.
So $A_k, A_{k+1}, \ldots$ are disjoint from $W$.
This contradicts to that  $W$ is a nbhd of $y$ and so infinitely many $A_n$ has to intersect $W$
because $y \in \limsup A_n$.
\end{proof}

An other example for a usc decomposition is the equivalence relation on $S^n$ defined by $x \sim -x$. Here the decomposition elements are not 
connected  and the decomposition space is the projective space $\RP^n$.
Or another example is the closed usc decomposition of $\R^2$, where the two non-singleton decomposition elements are 
the two arcs of the graph of the function $x \mapsto 1/x$, all the other decomposition elements are singletons.
The decomposition space is homeomorphic to 
$$A \cup_{\phi} B \cup_{\psi} A',$$
where $A$ and $A'$ are open disks, each of them   with one additional point  in its frontier denoted by
$a$ and $a'$ respectively.
The space $B$ is an open disk with two 
additional points $b, b'$ in its frontier and the gluing homeomorphisms are $\phi \co \{a\} \to \{b\}$ and 
$\psi \co \{a'\} \to \{b'\}$.
If a decomposition is given, then we would like to understand the decomposition space as well.

\begin{prop}
The decomposition space of a closed usc decomposition of a normal space is $T_4$.
\end{prop}
\begin{proof}
We have to show that if $\mathcal D$ is a usc decomposition of a normal space $X$, then any two disjoint closed sets  in the space
$X_{\mathcal D}$ can be separated by open sets. 
Let $A, B$ be disjoint closed sets in $X_{\mathcal D}$. Then $\pi^{-1}(A)$ and  $\pi^{-1}(B)$
are disjoint closed sets and by being $X$ normal and  by Lemma~\ref{saturated} they have disjoint saturated nbhds $U_1$ and $U_2$.
Taking $\pi(U_1)$ and $\pi(U_2)$ we get disjoint nbhds of $A$ and $B$. The decomposition elements are closed so $X_{\mathcal D}$ is $T_1$, which finally implies that $X_{\mathcal D}$ is $T_4$.
\end{proof}

If a space $X$ is not normal, then it is easy to define such closed usc decomposition, where the decomposition space
is even not $T_2$. Take two disjoint closed sets $A, B$ in $X$ which can not be separated by open sets. 
For example the direct product of the Sorgenfrei line with itself is not normal and choose the points with rational and irrational 
coordinates in the antidiagonal respectively, to have two closed sets $A$ and $B$.
These two sets are the two non-singleton elements of the decomposition $\mathcal D$, other elements are singletons.
Then $\mathcal D$ is closed usc but $X_{\mathcal D}$ is not $T_2$ because $\pi(A)$ and $\pi(B)$ can not be separated by open sets.

\begin{defn}
Let $\mathcal D$ be a decomposition of the space $X$. 
A decomposition  is \emph{finite} if it has only finitely many non-degenerate elements
and \emph{countable} if it has countably many
non-degenerate elements.
A decomposition is \emph{monotone} if  every decomposition element is  connected.
If $X$ is a metric space, then a decomposition is \emph{null} if 
the decomposition elements are bounded and
for every $\ep > 0$ there is only a finite number of elements whose diameter is greater than $\ep$.
\end{defn}

\begin{prop}
Let $\mathcal D$ be a decomposition and suppose that all elements are closed.
If $\mathcal D$ is finite, then it is a closed usc decomposition.
\end{prop}
\begin{proof}
Let $C \subset X$ be a closed subset, then $\pi^{-1}(\pi(C))$
is closed because
it is the finite union of the closed set $C$ and 
the  non-degenerate elements 
which intersect $C$.
Then by Lemma~\ref{saturated} (4) the statement follows.
\end{proof}

\begin{prop}
If $\mathcal D$ is a closed and null decomposition of a metric space, then it is usc.
\end{prop}
\begin{proof}
Denote the metric by $d$.
All the decomposition elements
are compact because
they are bounded.
Let $U$ be a nbhd of a $D \in \mathcal D$, then 
there is an $\ep >0$ such that
the $\ep$-nbhd of $D$ is in $U$.
Since $\mathcal D$ is null, there are only finitely many
decomposition elements $D_1, \ldots, D_n$
whose diameter is greater than $\ep/4$ and  $D_i \neq D$.
Let $\de$ be the minimum of $\ep/4$ and the distances
between 
$D$ and the $D_i$s.
If $D' \in \mathcal D$
is such that the distance between $D'$ and $D$ is less than $\de$, then
$D'$ is in the $\ep$-nbhd of $D$:
there are $x \in D$ and $y \in D'$ such that
$d(x, y) < \de$ so for every $a \in D'$
\begin{multline*}
\inf \{ d(a, b) :  b \in D \}
\leq
d(a, y) + d(y, x) + \inf \{ d(x, b) :  b \in D \}  
= d(a, y) + d(y, x) \leq \\ {\mathrm {diam}} \thinspace D' + \de
\leq \ep/2,
\end{multline*}
which means that 
$D'$ is in the $\ep$-nbhd of $D$ so $D' \subset U$.
\end{proof}

\begin{prop}
Let $\mathcal D$ be a usc decomposition of a space $X$.
\begin{enumerate}[\rm (1)]
\item
If $X$ is $T_2$, then $X_{\mathcal D}$ is $T_2$ as well.
\item
If $X$ is regular, then $X_{\mathcal D}$ is $T_3$.
\end{enumerate}
\end{prop}
\begin{proof}
The decomposition elements are compact so
every  $\pi^{-1}(a)$ and $\pi^{-1}(b)$ 
for different $a, b \in X_{\mathcal D}$ can be separated by open sets.
The statement follows easily.
\end{proof}

 \begin{prop}
Let $\mathcal D$ be a usc decomposition of a $T_2$ space $X$.
The decomposition $\mathcal D'$ whose elements are the connected components of the elements of $\mathcal D$
is a monotone usc decomposition.
\end{prop}
\begin{proof}
Take an element $D' \in \mathcal D'$ and denote by $D$ the decomposition element in 
$\mathcal D$ which contains $D'$. Suppose $D \neq D'$. 
Then $D - D'$ is closed in $D$ so  it is closed in $X$.
Let $U$ be a nbhd of $D'$.
Then there exists a nbhd $U' \subset U$ of $D'$ which is disjoint from a nbhd $U''$ of 
the closed set $D - D'$.
By the usc property we can find a
nbhd $V$ of 
$D$ such that $V  \subset U' \cup U''$ and 
if a $C \in \mathcal D$ intersects $V$, then $C \subset U' \cup U''$.
If $C' \in \mathcal D'$ intersects $V \cap U'$, then 
the element $C \in \mathcal D$ which contains 
$C'$ as a connected component
intersects
$V$ hence $C \subset U' \cup U''$.
Since $U'$ and $U''$ are disjoint, the component $C'$ of $C$
is in $U'$ because it intersects $U'$.
We got that $C' \subset U$.
\end{proof}

For example, it follows that the decomposition of a compact $T_2$ space $X$ whose elements are the connected components of the space is 
a usc decomposition. To see this, at first take the decomposition $\mathcal D$, where
$\mathcal H_{\mathcal D} = \{ X \}$ and hence the decomposition has no singletons. This is  usc so we can apply the previous 
proposition.

\begin{prop}
If $X$ is a metric space and $\mathcal D$ is its usc decomposition, then $X_{\mathcal D}$ is metrizable.
If $X$ is separable, then $X_{\mathcal D}$ is also separable.
\end{prop}
\begin{proof}
By \cite{St56} if there is a continuous closed map $f$ of a metric space onto a space $Y$ such that 
for every $y \in Y$
the closed set $f^{-1}(y) - {\mathrm {int}}\thinspace f^{-1}(y)$ is compact, then 
$Y$ is metrizable.
But for every $y \in X_{\mathcal D}$ the set $\pi^{-1}(y)$ and so its closed subset $\pi^{-1}(y) - {\mathrm {int}}\thinspace \pi^{-1}(y)$
are compact hence $X_{\mathcal D}$ is metrizable.
Moreover if $X$ is separable, then 
there is a countable subset $S \subset X$ intersecting every open set, 
which gives the countable set $\pi(S)$
intersecting every open set in $X_{\mathcal D}$.
\end{proof}

\bigskip{\Large{\section{Examples and properties of decompositions}}}
\bigskip\large

Usually, we are interested in the topology of the decomposition space
if a decomposition of $X$ is given.
Especially those situations are  stimulating where
the decomposition space turns out to be homeomorphic to $X$.

Let $X = \R$ and let $\mathcal D$ be a decomposition such that $\mathcal H_{\mathcal D}$ 
consists of countably many disjoint compact intervals. 
Then this is a  usc decomposition: any open interval $U \subset \R$ contains at most countably many compact intervals of $\mathcal H_{\mathcal D}$
and 
the infimum of the left endpoints of these intervals could be in $U$ or it could be the left boundary point of $U$. Similarly, we have this for the right 
endpoints. In all cases the union of the decomposition elements being in $U$ is open. 
For an arbitrary open set $U \subset \R$ we have the same, this means we have 
a usc decomposition.
Later we will see, that the decomposition space $X_{\mathcal D}$ is homeomorphic to $\R$. Moreover the decomposition map $\pi \co X \to X_{\mathcal D}$
is approximable by homeomorphisms, which means there are homeomorphisms from $\R$ to $\R$ arbitrarily close to $\pi$ in the sense of uniform 
metric. For example, let $X = \R$ and consider the infinite Cantor set-like construction by taking iteratively the middle third compact intervals
in the interval $[0,1]$. These are countably many intervals and 
define the decomposition ${\mathcal D}$ so that 
the non-degenerate elements are these intervals.
We can obtain this decomposition ${\mathcal D}$
by taking the connected components of $[0,1] - \mbox{Cantor set}$ and then  taking the closure of them.
This  is usc and we will see that the decomposition space is $\R$.

If  $X = \R^2$, then an analogous decomposition is  that $\mathcal H_{\mathcal D}$ consists of 
countably many compact line segments. More generally, let $\mathcal H_{\mathcal D}$ be countably many \emph{flat} arcs, that is
such subsets $A$ of $\R^2$ for which there exist  self-homeomorphisms $h_A$ of $\R^2$ mapping $A$ into the standard compact interval 
$\{ (x, 0) \in \R^2 : 0 \leq x \leq 1 \}$. Such a  decomposition is not necessarily usc, for example take the 
function $f \co [0,1) \to \R$, $f(x) = 1+x$, and the sequence $x_n = 1 -1/n$.
Define the decomposition by $\mathcal H_{\mathcal D} = \{ (x_n, y) : y \in [0, f(x_n) ], n \in \N \}$ and the singletons are
all the other points of $\R^2$. Then $\mathcal H_{\mathcal D}$ consists of countably many straight line segments but
this decomposition is not usc: consider the point $(1, 3/2) \in \mathcal D$ and its $\ep$-nbhds for small $\ep>0$. These 
intersect infinitely many non-degenerate
decomposition elements but none of the elements is a subset of  any of these $\ep$-nbhds.
The decomposition space is not $T_2$:
the points $\pi((1, y))$, where $0 \leq y \leq 2$,
cannot be separated by disjoint nbhds because 
the sequence $\pi((x_n, 0))$ converges to all of them.

However, if $\mathcal D$ is such a decomposition of $\R^2$ 
that $\mathcal H_{\mathcal D}$ consists of countably many {flat} arcs and further we suppose that $\mathcal D$ is usc, then the decomposition space 
$X_{\mathcal D}$ is homeomorphic to $\R^2$ and again $\pi$ can be approximated by homeomorphisms, we will see this later.

We get another interesting example by taking a smooth function with finitely many critical values on a closed manifold $M$. 
Then the decomposition elements are defined to be 
the connected components of the point preimages of the function. This is a monotone  decomposition ${\mathcal D}$ and it is usc because 
the decomposition map $\pi \co M \to M_{\mathcal D}$ is a closed map: in $M$ a closed set is compact, its $\pi$-image is compact as well and 
$M_{\mathcal D}$ is $T_2$ because it is a graph \cite{Iz88, Re46, Sa20} so this $\pi$-image is also closed.

If $X$ is $3$-dimensional, then the possibilites increase tremendously. This is illustrated by the following surprising statement.
\begin{prop}\label{line_decomp}
For every compact metric space $Y$ there exists a monotone usc decomposition of the compact ball $D^3$ such that $Y$ can be embedded into the  decomposition space. 
\end{prop}
\begin{proof}
Recall that by the Alexandroff-Hausdorff theorem the Cantor set in the $[0,1]$ interval can be mapped surjectively and continuously
onto every compact metric space.
Let $T$ be a tetrahedron in $D^3$, denote two of its non-intersecting edges by 
$e$ and $f$. Identify these edges linearly with $[0,1]$ and let $C_1$ and $C_2$ be the Cantor sets in $e$ and $f$, respectively.
For $i=1, 2$ denote the existing surjective maps of $C_i$ onto $Y$ by 
$\psi_i \co C_i \to Y$.
For every $x \in Y$ take the union of all the line segments in $T$ connecting 
all the points of $\psi_1^{-1}(x)$ to all the points of  $\psi_2^{-1}(x)$.
Denote this subset of $T$ by $D_x$, see Figure~\ref{tetrahedron}. They are compact and connected for all $x \in Y$ and they are pairwise disjoint because
all the lines in $T$ connecting points of $e$ and $f$ are pairwise disjoint.
So we have a monotone usc decomposition with 
$\mathcal H_{\mathcal D} = \{ D_x : x \in Y \}$.
Define the embedding $i$ of $Y$ into $D^3_{\mathcal D}$ by
$i (x ) =    \pi( \psi_1^{-1}(x))$.
This map is injective, closed because $\pi$ is closed and continuous because 
$\psi_1$ is closed.
\end{proof}

\begin{figure}[h!]
\begin{center}
\epsfig{file=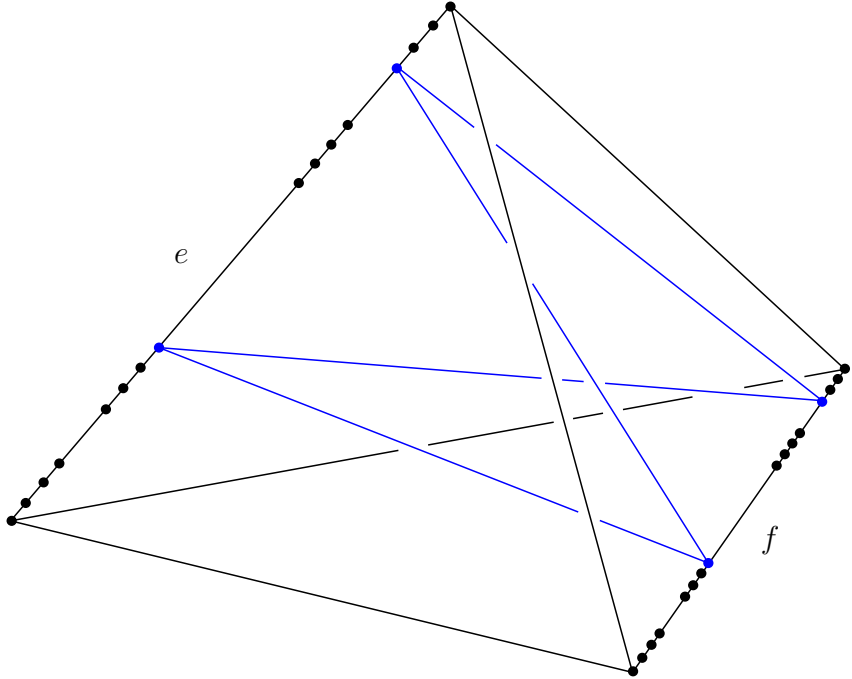, height=9cm}
\put(-1.2, 1.7){$f$}
\put(-9, 5.5){$e$}
\end{center} 
\caption{The tetrahedron $T$, the edges $e$ and $f$ and a set $D_x$ pictured in blue.}
\label{tetrahedron}
\end{figure}

To see further examples in $\R^3$ let us introduce some notions.

\begin{defn}[Defining sequence]
Let $X$ be a connected  $n$-dimensional manifold. A \emph{defining sequence} for a decomposition 
of $X$ 
 is a sequence $$C_1, C_2, \ldots, C_n, \ldots$$ of compact  $n$-dimensional 
 submanifolds-with-boundary 
 in $X$ such that
 $C_{n+1} \subset \interior C_n$.
The decomposition  elements of the defined decomposition
are the connected components of $\cap_{n=1}^{\infty} C_n$ and the other points of $X$ are singletons.
\end{defn}

Obviously a decomposition defined in this way is monotone. 
The set $\cap_{n=1}^{\infty} C_n$ is closed and compact so its connected components are closed and compact as well.
Also the space $\cap_{n=1}^{\infty} C_n$ is $T_2$ hence its decomposition to its connected components is usc.
Then adding all the points of $X - \cap_{n=1}^{\infty} C_n$ to this decomposition  as singletons results our decomposition.
This is usc:
the only thing which is not completely obvious is that 
in a nbhd of an added point the conditions being usc are satisfied or not. But $\cap_{n=1}^{\infty} C_n$ is closed, its complement is open so every such singleton 
has a nbhd disjoint from $\cap_{n=1}^{\infty} C_n$.

\begin{prop}
If all $C_n$ in a defining sequence is connected, then  $\cap_{n=1}^{\infty} C_n$ is connected.
\end{prop}
\begin{proof}
Let $C$ denote the non-empty set $\cap_{n=1}^{\infty} C_n$.
Suppose $C$ is not connected, this means there are disjoint closed non-empty subsets $A, B \subset C$ 
such that $A \cup B = C$. These $A$ and $B$ are closed in the ambient manifold $X$ as well, so there exist disjoint
nbhds $U$  of $A$ and  $V$ of $B$ in $X$.
It is enough to show that for some $n \in \N$ we have  $C_n \subset U \cup V$, because then 
$C_n \cap U \neq \emptyset$, $C_n \cap V \neq \emptyset$ imply that
$C_n$ is not connected, which is a contradiction.
If we suppose that for every $n \in \N$ we have $C_n \cap (X - (U \cup V)) \neq \emptyset$, then 
for every $n$ we have $C_n \cap (X - U) \cap (X-V) \neq \emptyset$, i.e.\ the closed set
$F = (X - U) \cap (X-V)$ and each element of the nested sequence $C_1, C_2, \ldots$ satisfy 
$$C_n \cap F \neq \emptyset.$$
Of course
$$C_{n+1} \cap F \subset C_n \cap F$$
which implies that 
$$F \cap C = F \cap (\cap_{n=1}^{\infty} C_n) =  \cap_{n=1}^{\infty} (C_n \cap F)  \neq \emptyset$$
because
all $C_n \cap F$ is closed in the compact space $C_1$.
But $F \cap C \neq \emptyset$ contradicts to $C \subset U \cup V$.
\end{proof}

The $\pi$-image of the union of non-degenerate elements of a decomposition associated to a defining sequence
is closed and also totally disconnected because if $\cap_{n=1}^{\infty} C_n$ is not connected, then
all the pairs of decomposition elements have disjoint saturated nbhds which yield 
disjoint nbhds of their $\pi$-image.

\subsection{The Whitehead continuum}

One of the most famous such decomposition is related to the so called Whitehead continuum. Its defining sequence consists of solid tori
embedded into each other in such a way that $C_{i+1}$ is a thickened Whitehead double of the center circle of $C_i$, see Figure~\ref{whitehead_decomp}.
The intersection $\cap_{i = 1}^{\infty} C_i$ is a compact subset of $\R^3$, this is the Whitehead continuum, which we denote by
$\mathcal W$. The decomposition consists of the connected 
components of $\mathcal W$ and the singletons in the complement of them. 
If the diameters $d_i$ of the meridians of the tori $C_i$ converges to $0$ as $i$ goes to $\infty$, then
$\mathcal W$ intersects the vertical sheet $S$ in Figure~\ref{whitehead_decomp}
in a Cantor set: $C_i \cap S$ is equal to $2^{i-1}$ copies of disks of diameter $d_i$ nested into each other. The intersection 
$S \cap (\cap_{i = 1}^{\infty}  C_i)$ is then a Cantor set. 
The Whitehead continuum $\mathcal W$ is connected because the $C_i$ tori are connected 
but it is not path-connected.  
We will see later that the decomposition space $\R^3_{\mathcal W}$  is not homeomorphic to $\R^3$ but taking its direct product with $\R$ we get  $\R^4$.
An important property of $\R^3 - \mathcal W$ is that it is a contractible $3$-manifold, which is not homeomorphic to $\R^3$.

For understanding further properties of this decomposition, we are going to define some notions.

\begin{figure}[h!]
\begin{center}
\epsfig{file=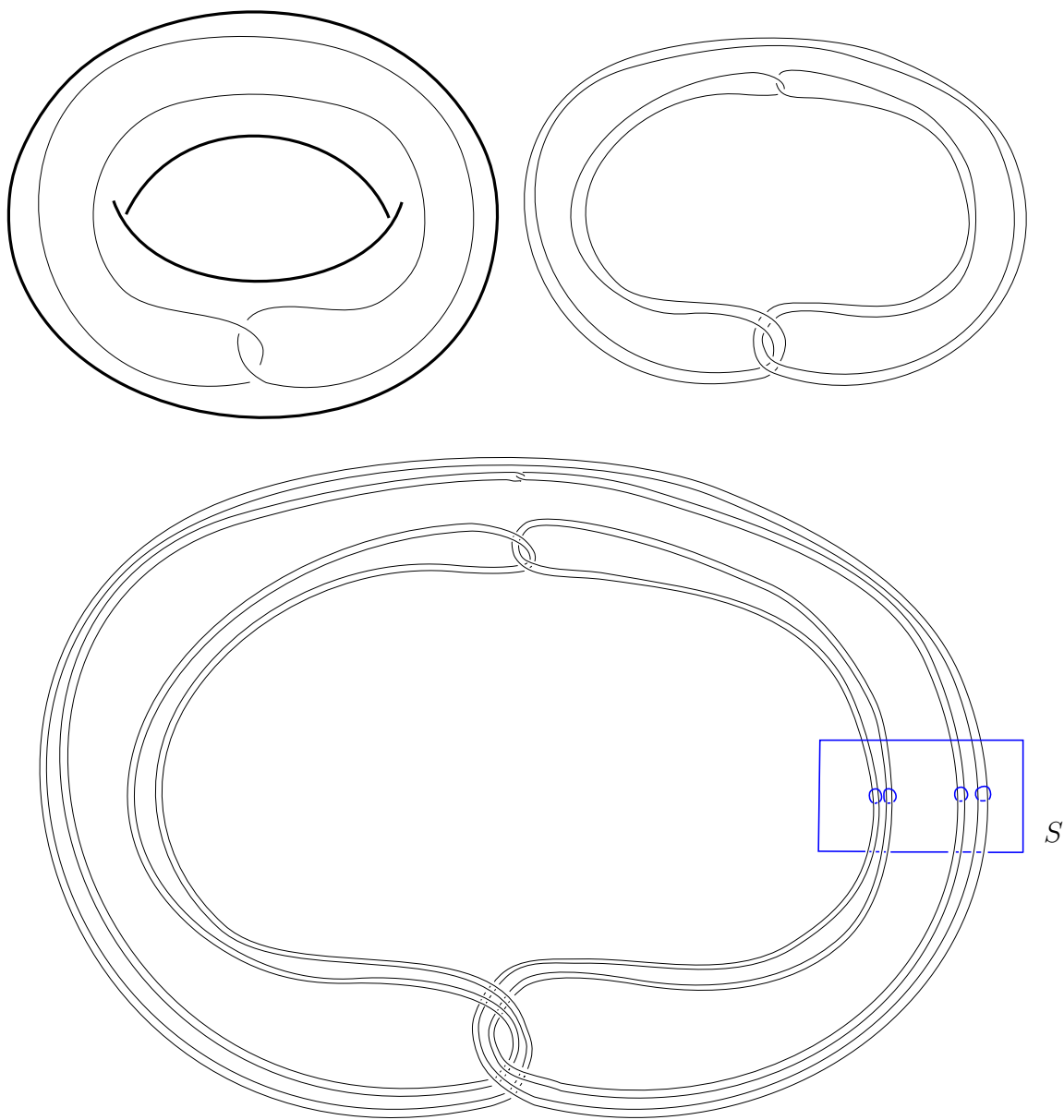, height=16cm}
\put(0.2, 4){$S$}
\end{center} 
\caption{A sketch of the defining sequence of the Whitehead decomposition. The first figure shows the solid torus $C_1$ and the Whitehead double of its 
center circle. The second figure shows the Whitehead double of the center circle of $C_2$. The torus $C_2$ is not shown but we get it
by thickening the Whitehead double in $C_1$. Then thicken the knot in the second figure (so we get the solid torus $C_3$) and take its center circle.
Take the Whitehead double of this circle and so we get the knot embedded in $C_3$ in the third figure.
In the third figure we can see the intersection of $C_3$ with a vertical sheet $S$,
which is four small disks. This vertical sheet $S$ intersects the Whitehead continuum in a Cantor set.}
\label{whitehead_decomp}
\end{figure}

%


\begin{defn}[Cellular set, cell-like set]
Let $X$ be an $n$-dimensional manifold and $C \subset X$ be a subset of $X$. 
The set $C$ is \emph{cellular} 
if there is a sequence $B_1, B_2, \ldots, B_n, \ldots$ of closed  $n$-dimensional 
balls in $X$ such that
$B_{n+1} \subset \interior B_n$
and $C = \cap_{n=1}^{\infty} B_n$.
A compact subset $C$ of a topological space $X$  is \emph{cell-like} if 
for every nbhd $U$ of $C$ there is a nbhd $V$ of $C$ in $U$ such that
the inclusion map $V \to U$ is homotopic in $U$ to a constant map.
Similarly, a decomposition is called cellular or cell-like if
each of its decomposition elements is cellular or cell-like, respectively.
\end{defn}

For example the ``topologist's sine curve'' in $\R^2$
is cellular.
A cellular set is compact and also connected but not necessarily path-connected.
It is also easy to see that
every compact contractible subset of a manifold is cell-like.
Also a compact and contractible metric space is cell-like in itself.
A  cell-like set $C$ is connected because if  there were
two  open subsets $U_1$ and $U_2$ in $X$ separating some connected components of $C$, then 
it would be not possible to contract any nbhd $V \subset U_1 \cup U_2$ of $C$ to one single point.

\begin{prop}\label{Wnotcell}
The set $\mathcal W$ is cell-like but not cellular.
\end{prop}
\begin{proof}
Let $U$ be a nbhd of $\mathcal W$. Then there is an $n$ such that 
$C_i \subset U$ for all $i \geq n$.
Let $V$ be such a small tubular nbhd of $C_{n+1}$ which is inside $C_n$. Then since the Whitehead double of the center
circle of $C_n$ is null-homotopic in the solid torus $C_n$,  the thickened Whitehead double $C_{n+1}$ and its nbhd 
$V$ are  also null-homotopic in  $C_n$, hence the map $V \to U$ is homotopic in $U$ to a constant map.

\begin{lem}
The $3$-manifold $S^3 - \mathcal W$ is not simply connected at infinity.
\end{lem}
\begin{proof}
We have to show that there is a compact subset $C \subset S^3 - \mathcal W$ such that 
for every compact set $D \subset S^3 - \mathcal W$ containing 
$C$ the induced homomorphism 
$$
\varphi \co \pi_1(S^3 - \mathcal W - D) \to \pi_1(S^3 - \mathcal W - C)
$$
is not the zero homomorphism.
Let $C$ be the closure of 
$S^3 - C_1$.
If $D$ is a compact set in $S^3 - \mathcal W$ containing 
$C$, then $S^3 - D$ is a nbhd of $\mathcal W$
in $C_1$. Then there is an $n$ such that 
$C_i \subset S^3 - D$ for all $i \geq n$.
Consider the commutative diagram
\begin{equation*}
\begin{CD}
\pi_1(S^3 - C_n - D) @>>> \pi_1(S^3 - C_n - C)    \\
@VVV  @VV \alpha V \\
\pi_1(S^3 -  \mathcal W - D) @> \varphi >> \pi_1(S^3 -  \mathcal W - C). 
\end{CD}
\end{equation*}
By \cite{NW37} 
the generator of the group $\pi_1(S^3 - C_n - C)$ 
represented by
the meridian of the torus $C_n$ is mapped by $\alpha$
into
a generator of $\pi_1(S^3 -  \mathcal W - C)$.
Since this meridian  also represents an element of
$\pi_1(S^3 - C_n - D)$, we get that 
$\varphi$ is not the zero homomorphism.
\end{proof}
Let us continue the proof of Proposition~\ref{Wnotcell}. 
If $\mathcal W$ is cellular, then there are 
$B_1, B_2, \ldots, B_n, \ldots$ closed  $n$-dimensional 
balls in $S^3$ such that
$B_{n+1} \subset \interior B_n$
and $\mathcal W = \cap_{n=1}^{\infty} B_n$.
This would imply that 
$S^3 -  \mathcal W$ is simply connected at infinity because
if $C \subset S^3 -  \mathcal W$  is a compact set, then
take a $B_n \subset S^3  - C$ and
a loop in $\mathrm {int} B_n - \mathcal W$, then
there is a $B_m \subset \mathrm {int}  B_n$ not containing this loop 
and the loop in null-homotopic 
in $\mathrm {int}  B_n-B_m$ because
$\pi_1 (\mathrm {int}  B_n-B_m ) = 0$. 
Hence we obtain that $S^3 -  \mathcal W$ is not cellular.
\end{proof}
With more effort we could show that 
$S^3- \mathcal W$
is contractible
so it is homotopy equivalent to $\R^3$ but 
by the previous statement it is not
homeomorphic to $\R^3$.
It is known that
the set
$\mathcal W \x \{ 0 \}$ is cellular in $\R^3 \x \R$ and the decomposition space 
of the decomposition of 
 $\R^3 \x \R$ whose only
 non-degenerate element is 
 $\mathcal W \x \{ 0 \}$
 is homeomorphic to $\R^4$.
 This fact is the starting point  
 of the proof of the $4$-dimensional Poincar\'e conjecture.

Being cell-like often does not depend on the ambient space.
To understand this, we have to introduce a new  notion.

\begin{defn}[Absolute nbhd retract]
A metric space $Y$ is an \emph{absolute nbhd retract}  (or \emph{ANR} for short)
if for an arbitrary metric space $X$ and its closed subset $A$  every map $f$ from $A$ to $Y$ 
extends to a nbhd of $A$. In other words, the nbhd $U$ and the dashed arrow exist in the following diagram 
and make the diagram commutative.
\begin{center}
\begin{graph}(6,4.5)
\graphlinecolour{1}\grapharrowtype{2}
\textnode {A}(0.5,1.5){$A$}
\textnode {X}(5.5, 1.5){$X$}
\textnode {U}(3, 0){$U$}
\textnode {Y}(5.5, 4){$Y$}
\diredge {A}{Y}[\graphlinecolour{0}]
\diredge {A}{U}[\graphlinecolour{0}]
\diredge {U}{X}[\graphlinecolour{0}]
\diredge {A}{X}[\graphlinecolour{0}]
\diredge {U}{Y}[\graphlinecolour{0}\graphlinedash{4}]
\freetext (3,3.2){$f$}
\freetext (3,1.2){$\subseteq$}
\freetext (1.2, 0.6){$\subseteq$}
\freetext (4.8, 0.6){$\subseteq$}
\end{graph}
\end{center}
\end{defn}

This is equivalent to say that 
for every metric space $Z$ and embedding $i \co Y \to Z$  such that 
$i(Y)$ is closed there is a nbhd $U$ of $i(Y)$ in $Z$ which retracts onto $i(Y)$, that is
$r|_{i(Y)} = \mathrm {id}_{i(Y)}$ for some map $r \co U \to i(Y)$.
It is a fact that every manifold is an ANR.

The property of cell-likeness is independent of the ambient space until that is an ANR as
the following statement shows.

\begin{prop}
If $C \subset X$ is a compact cell-like set
in a metric space $X$, then
the embedded image of $C$ in an arbitrary ANR is also cell-like.
\end{prop}
\begin{proof}
Suppose $e \co C \to Y$ is an embedding into an ANR $Y$.
We have to show that $e(C)$ is cell-like.
Let $U$ be a nbhd of $e(C)$. Since $Y$ is ANR, there is a nbhd $\tilde V$ 
of $C$ in $X$ such that $e$ extends to an $\tilde e \co \tilde V \to Y$.
Let $V \subset X$ be the open set $\tilde V \cap \tilde e^{-1}(U)$, it is a nbhd of $C$.
There is a nbhd $W$ of $C$ such that $C \subset W \subset V$ and
there is a homotopy 
of the inclusion $W \subset V$ to the constant in $V$ since $C$ is cell-like, denote this homotopy by
$\varphi \co W \x [0,1] \to V$.
Then $\varphi|_{C \x [0,1]}$ is a homotopy of the inclusion $C \subset V$ to the constant.
Take 
$$\tilde e \circ \varphi|_{C \x [0,1]} \circ (e^{-1}|_{e(C)} \x \mathrm {id}_{[0,1]}),$$ this
is a homotopy of the inclusion $e(C) \subset U$ to the constant in $U$.
The space $e(C) \x [0,1]$
is compact in $Y \x [0,1]$ and the homotopy maps it into $Y$, which is ANR.
This implies that
there is a nbhd $\tilde U \subset U$ of $e(C)$ such that the inclusion $\tilde U \subset U$
is homotopic to constant in $U$.
\end{proof}
For example, this shows that a compact and contractible metric space is cell-like
if we embed it into any ANR. 
In practice, we 
do not consider cell-like sets as subsets in some ambient space but rather
as compact metric spaces which are cell-like if we embed them into an arbitrary ANR.

It is clear that every cellular set $C$ is cell-like because in every nbhd $U$ of $C$ some open ball  is contractible. 
Also, we have seen that the Whitehead continuum is cell-like but not cellular.
In order to compare cell-like and cellular sets we introduce the notion of
cellularity criterion.

\begin{defn}[Cellularity criterion]
A subset $Y \subset X$ satisfies the \emph{cellularity criterion}
if for every nbhd $U$ of $Y$ there is a nbhd $V$ of $Y$ such that $V \subset U$
and every loop in $V - Y$ is null-homotopic in $U - Y$.
\end{defn}

The cellularity criterion and being cellular measure how wildly a subset is embedded into a space.
The next theorem compares cell-like and cellular sets in a PL manifold. We omit its difficult proof here.
\begin{thm}
Let $C$ be a cell-like subset of a PL $n$-dimensional manifold, where $n \geq 4$.
Then $C$ is cellular if and only if $C$ satisfies the cellularity criterion.
\end{thm}

In dimension $2$ we have a simpler statement:
\begin{thm}
Every cell-like subset in a $2$-dimensional manifold $X$ is cellular.
\end{thm}
\begin{proof}
At first suppose $X = \R^2$ and $C \subset \R^2$ is a cell-like set.
Let $U$ be a bounded nbhd of $C$ and let $V \subset U$ a nbhd of $C$ such that 
the inclusion $V \to U$ is homotopic to constant.
Choose another nbhd $W \subset V$ of $C$ as well such that $\mathrm{cl} \thinspace W \subset V$.
 Take a compact smooth $2$-dimensional manifold
$H \subset V$ such that $C \subset \mathrm{int} H$, $\del H \subset V - \mathrm{cl} \thinspace W$ and $\mathrm{int} H$ is connected.
Such an $H$
can be obtained by 
taking a Morse function $f \co V \to [0,1]$ which
maps the nbhd $W$ of $C$ into $0$ and a small nbhd of $\R^2 - V$ into $1$.
Then the preimage of a regular value $r$ close to $1/2$ is a smooth $1$-dimensional submanifold of $\R^2$
and the preimage of $(-\infty, r]$ is a compact subset containing
$W$ and $C$, denote this $f^{-1}((-\infty, r])$ by $H$.  
Then $H$ is
a compact smooth $2$-dimensional submanifold of $\R^2$, see Figure~\ref{twomanifoldconstruct}. 
Take its connected component (this is also a path-connected component because
$H$ is a manifold)
which contains $C$ and denote this by $H$ as well.
\begin{figure}[h!]
\begin{center}
\epsfig{file=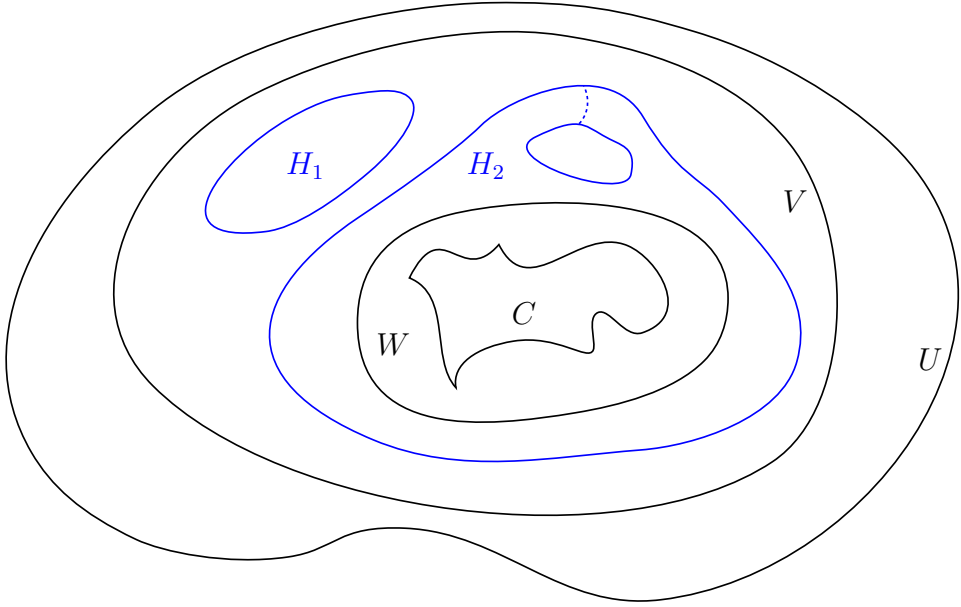, height=8cm}
\put(-6, 3.7){$C$}
\put(-0.6, 3.1){$U$}
\put(-2.4, 5.2){$V$}
\put(-7.8, 3.3){$W$}
\put(-6.6, 5.7){${\textcolor[rgb]{0,0,1}{H_2}}$}
\put(-9, 5.7){${\textcolor[rgb]{0,0,1}{H_1}}$}
\end{center}
\caption{The compact manifold $H = H_1 \cup H_2$. Its component $H_2$ contains $C$.
Since $H_2 - C$ is path-connected, there
is a path (dashed in the figure) in $H_2$ connecting two different components of the boundary of $H_2$. 
}
\label{twomanifoldconstruct}
\end{figure}

We show that $H - C$ is connected.
For this consider the commutative diagram
\begin{equation*}
\begin{CD}
H_1 ( H; \Z_2 )  @>>> H_1 ( H, H - C; \Z_2 ) @>>> H_0 (  H - C; \Z_2 )  @>>> H_0 ( H ) @>>> 0 \\
@VVV @VVV @VVV @VVV \\
H_1 ( \R^2 ; \Z_2 )  @>>> H_1 ( \R^2, \R^2 - C; \Z_2 ) @>>> H_0 (  \R^2 - C; \Z_2 )  @> i_*>> H_0 ( \R^2 ) @>>> 0 
\end{CD}
\end{equation*}
coming from the long exact sequences and the inclusion $( H, H - C ) \subset ( \R^2, \R^2 - C )$.
This is just the diagram
\begin{equation*}
\begin{CD}
H_1 ( H; \Z_2 )  @>>> H_1 ( H, H - C; \Z_2 ) @>>> H_0 (  H - C; \Z_2 )  @>>>  \Z_2  @>>> 0 \\
@VVV @VV \cong V @VVV @VV \cong V \\
0  @>>> H_1 ( \R^2, \R^2 - C; \Z_2 ) @>>> H_0 (  \R^2 - C; \Z_2 )  @> i_* >> \Z_2  @>>> 0  
\end{CD}
\end{equation*}
If the group $H_0 (  \R^2 - C; \Z_2 )$ is $\Z_2$, i.e.\ the manifold $\R^2 - C$ is connected, then 
exactness implies that $H_0 (  H - C; \Z_2 ) \cong \Z_2$
so $H - C$ is connected.
To show that $\R^2 - C$ is connected, we apply
\cite[Theorem~VI.5, page~86]{HW41}, which implies that
if $C$ is a closed subset of a space $D$ and $f, g$ are homotopic maps of $C$ into $S^1$ such that 
$f$  extends to $D$, then $g$ extends to $D$ and the extensions are homotopic.
Suppose the open set $\R^2 - C$ is not connected, then 
it is the disjoint union of two open sets $A$ and $B$. At least one of these is bounded
because for large enough $s$ the set $\R^2 - [-s, s]^2$ is disjoint from $C$ and it is connected 
hence it is in $A$ or $B$ but then $[-s, s]^2$ contains $B$ or $A$, respectively.
Suppose $A$ is bounded, $p \in A$ and $q \in B$.
For a subset $S \subset \R^2$ and point $x \in \R^2$
denote by $\pi_{S,x} \co S - \{x\} \to S^1$ the radial projection of $S - \{x\}$ to the circle $S^1$ of radius $1$  centered at $x$.
Then $\pi_{C,q}$ extends to $\R^2 - \{q\}$ so also to $A \cup C$ but
$\pi_{C,p}$ does not extend to $A \cup C$
because
such an extension
would extend to a much
  larger disk $P$ centered at $p$ as well by radial projection and then a retraction 
 of $P$ onto its boundary (if we identify it with the target circle of $\pi_{C,p}$) would exists.
 Consequently  $\pi_{C,q}$ and $\pi_{C,p}$ are not homotopic and so 
 at least one of them is not homotopic to constant.
This means if $\R^2 - C$ is not connected, then 
there is a map $C \to S^1$ which is not homotopic to constant.
But since the inclusion $V \subset U$ and then also $C \subset U$ are homotopic to constant,
we get that $\R^2 - C$ is connected. 

Finally, we get that $H - C$ is a  path-connected smooth $2$-dimensional manifold
with boundary. Hence if the number of components of $\del H$ is larger than one, then 
there exists a smooth curve transversal to $\del H$, disjoint from $C$
and connecting different components of $\del H$. We can cut $H$ along this curve and by 
repeating this process we end up with $\del H$ being a single circle.
By the Jordan curve theorem 
$H$ is a compact $2$-dimensional disk. 
In this way we get $$C \subset W \subset \mathrm { int} H \subset H \subset V \subset U.$$
Since
in $\R^2$ every compact set is a countable intersection of open sets which form a decreasing sequence,
we have 
$C = \cap_{n = 1}^{\infty} U_n$, where $U_1 \supset U_2 \supset \cdots \supset U_n \supset \cdots$,
where the sets $U_n$ are open. 
We can also assume
that for each $n$ we have $\cl U_{n+1} \subset U_n$.

We obtain countably many compact $2$-dimensional disks 
$H_1, H_2, \ldots$ by the previous construction, which satisfy
$$C \subset U_{n+1} \subset \mathrm { int} H_n \subset H_n \subset V \subset U_n.$$
Hence $C = \cap_{n = 1}^{\infty} H_n$ so $C$ is cellular.

In the case of $X$ is an arbitrary $2$-dimensional manifold, 
 since $C$ is cell-like, there exists a nbhd of $C$ which is homotopic to constant 
so $C$ is contained in a simply-connected $2$-dimensional manifold nbhd, which is homeomorphic to $\R^2$.
Hence a similar argument gives that $C$ is cellular.
\end{proof}

\begin{prop}
If $C$ is cell-like in a smooth $n$-dimensional manifold $X$, where $n \geq 3$, then 
$C \x \{ 0 \}$ is cellular in $X \x \R^3$.
\end{prop}
\begin{proof}
It is enough to show that $C \x \{ 0 \}$ satisfies the cellularity criterion.
It is easy to see that 
$C  \x \{ 0 \}$ is cell-like in $X \x \R^3$.
Let $U$ be a nbhd of $C  \x \{ 0 \}$ in $X \x \R^3$.
It is obvious that 
 there is a nbhd $V \subset U$ of $C \x \{ 0 \}$ 
such that  every 
loop $\ga \co [0,1] \to V$
is null-homotopic in $U$.
Let $\ga$ be an arbitrary loop in $V - C \x \{ 0\}$, it is homotopic to a smooth loop in $\tilde \ga \co V - C \x \{ 0\}$ by a homotopy $H$.
A homotopy of $\tilde \ga$ to constant  can be approximated by
a smooth map $\tilde H \co D^2 \to U$, where $\tilde H|_{\del D^2} = \tilde \ga$. 
 In the subspace $X \x \{ 0 \}$ of $X \x \R^3$ 
 let $W$ be a nbhd of $C \x \{ 0 \}$
 which is disjoint from 
the  homotopy $\tilde H$.
Perturb $\tilde H$ keeping $\tilde H|_{\del D^2}$ fixed  to get a transversal map to 
the $n$-dimensional manifold $W$ in $U$, hence we get that 
$\ga$ is null-homotopic in $U - C \x \{ 0 \}$.
So the cellularity criterion holds for $C \x \{ 0 \}$.
\end{proof}

\subsection{Antoine's necklace}

Take the defining sequence where 
\begin{itemize}
\item
$C_1$ is a solid torus, 
\item
$C_2$ is a finite number of solid tori embedded in $C_1$ in such a way that
each torus is unknotted and linked to its neighbour as in a usual chain, 
\item
$C_3$ is again a finite number of similarly linked solid tori, 
\end{itemize}
\ldots, etc., see Figure~\ref{antoine}.

\begin{figure}[h!]
\begin{center}
\epsfig{file=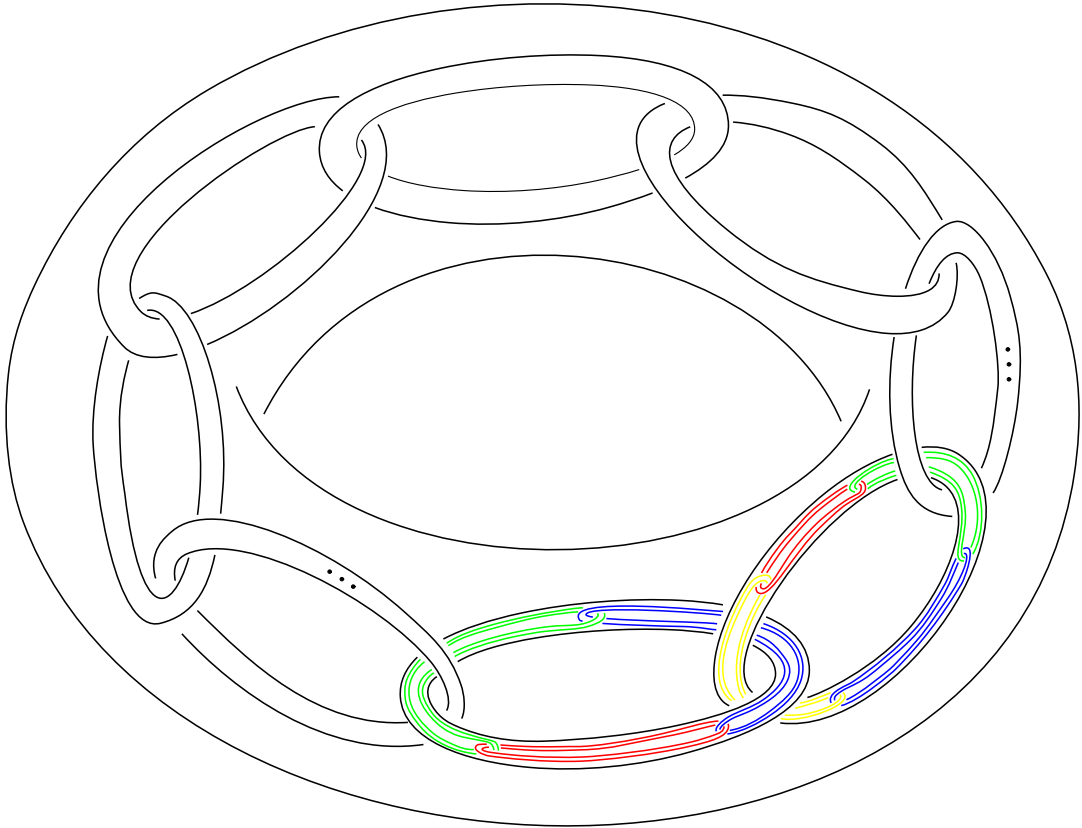, height=11cm}
\end{center} 
\caption{A sketch of the defining sequence of Antoine's necklace. We can see the solid torus $C_1$, the linked tori $C_2$ and some linked tori from the collection $C_3$, etc.
The number of components of $C_{n+1}$ in $C_n$ 
is large enough to make the diameters of the tori converge to $0$.}
\label{antoine}
\end{figure}

We always consider at least three tori in each $C_n$.
We require that the maximal diameter of tori in $C_n$ converges to $0$.
The set $\cap_{n = 1}^{\infty} C_n$ is called Antoine's necklace and denoted by $\mathcal A$.
It is easy to see that each of its components is cell-like. 
Unlike Whitehead continuum the components of $\mathcal A$
are cellular because every component of $C_{n+1}$ is inside a ball in $C_n$.

Recall that the Cantor set is the topological space 
$$D_1 \x D_2 \x \cdots \x D_n \x \cdots,$$
where every space $D_n$ is a finite discrete metric space with $|D_n| \geq 2$.

\begin{prop}
The space $\cap_{n = 1}^{\infty} C_n$
is homeomorphic to the Cantor set.
\end{prop}
\begin{proof}
Denote the number of tori embedded in $C_1$  by $m_1$, these tori are $$C_{2, 1}, \ldots, C_{2, m_1}$$
whose disjoint union is $C_2$.
For $1 \leq i_1 \leq m_1$ take the $i_1$-th torus $C_{2, i_1}$ and 
denote the number of tori embedded into it  by $m_{2, i_1}$, these tori are
$$C_{3, i_1, 1}, \ldots, C_{3, i_1, m_{2, i_1}}$$
whose disjoint union is $C_3$.
 Again for $1 \leq i_2 \leq m_{2, i_1}$
take the $i_2$-th torus $C_{3, i_1, i_2}$  and 
denote the number of tori embedded into it by $m_{3, i_1, i_2}$,
 these tori are
$$C_{4, i_1, i_2, 1}, \ldots, C_{4, i_1, i_2, m_{3, i_1, i_2}}$$
whose disjoint union is $C_4$.
In general in the $n$-th step
for $1 \leq i_n \leq m_{n, i_1, \ldots, i_{n-1}}$
take the $i_n$-th torus
$C_{n+1, i_1, \ldots, i_n}$
 and
denote the number of tori embedded into it by $m_{n+1, i_1, \ldots, i_n}$, 
 these tori are
$$C_{n+2, i_1, \ldots, i_n, 1}, \ldots, C_{n+2, i_1, \ldots, i_n, m_{n+1, i_1, \ldots, i_n}}$$
whose disjoint union is $C_{n+2}$.

Now we construct a Cantor  set $\mathcal C$ in the interval $[0,1]$. 
Divide $[0,1]$  into $2m_{1} -1$ closed intervals 
$$I_{2, 1}, \ldots, I_{2, 2m_{1} -1} \subset [0,1]$$
of equal length and disjoint interiors.
Then divide the $i_1$-th interval $I_{2, i_1}$, where $i_1$ is odd,
into $2m_{2, i_1} -1$ closed intervals
$$I_{3, i_1, 1}, \ldots, I_{3, i_1, 2m_{2, i_1}-1}$$ of equal length.
Then divide the $i_2$-th interval $I_{3, i_1, i_2}$, where $i_2$ is odd,
into $2m_{3, i_1, i_2} -1$ closed intervals
$$I_{4, i_1, i_2, 1}, \ldots, I_{4, i_1, i_2, 2m_{3, i_1, i_2}-1}$$
of equal length.
In the $n$-th step divide the $i_n$-th interval $I_{n+1, i_1, \ldots, i_n}$, where $i_n$ is odd,
into the closed intervals
$$I_{n+2, i_1, \ldots, i_n, 1}, \ldots, I_{n+2, i_1, \ldots, i_n,  2m_{n+1, i_1, \ldots, i_n}-1}$$
of equal length and so on.
So all the intervals $I_{n+1, i_1, \ldots, i_n }$ have length 
$$\frac{1}{(2m_{1} -1)\cdots(2m_{n, i_1, \ldots, i_{n-1}}-1)}.$$

Then let
$$\mathcal C = \bigcap_{n=1}^{\infty} \bigcup_{
{\begin{smallmatrix}
1 \leq i_1 \leq m_1    \\
 1 \leq i_2 \leq m_{2, i_1}   \\
\cdots \\
   1 \leq i_n \leq m_{n, i_1, \ldots, i_{n-1}} 
\end{smallmatrix}}
} I_{n+1, 2i_1-1, \ldots, 2i_n-1 }.$$
Assign to a point $x \in \cap_{n = 1}^{\infty} C_n$  the point $$\bigcap_{n=1}^{\infty} I_{n+1, 2i_1(x)-1, \ldots, 2i_n(x)-1 },$$ which is the
intersection
of the closed intervals containing $x$. 
This defines a map $$f \co \cap_{n = 1}^{\infty} C_n \to \mathcal C,$$ which is clearly
surjective.
It is injective as well because 
if $x \neq x'$, then 
for large $n$ they are in different $C_n$ so they are mapped into different intervals as well.
The map $f$ is continuous because
 if $x$ and $x'$ are in the same $C_n$ until some large enough $n$, then 
they are mapped to the same intervals until a large index so $f(x)$ and $f(x')$ are 
close enough. Then $f$ is a homeomorphism since its domain is compact and it maps injectively into a $T_2$ space.
\end{proof}

Of course the components of $\mathcal A$
are points so the decomposition space is obviously $\R^3$.
An important property of $\mathcal A$ is that it is \emph{wild}, i.e.\ 
there is no self-homeomorphism of $\R^3$ 
mapping $\mathcal A$ onto the  Cantor set in a  line segment.
To prove this, we study the local behaviour of the complement of $\mathcal A$.
\begin{defn}
Let $k \geq 0$.
A closed subset $A$ of a space $X$ is locally
$k$-co-connected ($k$-LCC for short) if for every point $a \in A$ and for every  nbhd $U$ of $a$ in $X$
there is a nbhd $V \subset U$ of $a$ in $X$ such that
if $\varphi \co \del D^{k+1} \to V-A$ is a map of the $k$-sphere, then
$\varphi$ extends 
to a map of $D^{k+1}$ into $U- A$.
\end{defn}

\begin{prop}
The set $\mathcal A$ in $\R^3$ is not $1$-LCC.
\end{prop}
\begin{proof}[Sketch of the proof]
At first we show that 
if $\alpha \co S^1 \to C_1$ is the meridian of the 
torus $C_1$, then
every smooth embedding $\tilde \alpha \co D^2 \to \R^3$ extending $\alpha$ is such that $\tilde \alpha( D^2)$ intersects $\mathcal A$.
If this was not true, then
$\tilde \alpha( D^2)$ would intersect at most finitely many tori $C_1, \ldots, C_n$
and it would be possible to perturb $\tilde \alpha$ to get a smooth embedding  transversal to each $\del C_n$.
Then it is possible to show that there is a disk $D_1 \subset D^2$ such that 
$\tilde \alpha (\del D_1)$ intersects some torus $\del C_{2, i_1}$
in a meridian. Inductively, $\tilde \alpha( D^2)$ has to intersect some torus $\del C_{m, i_1, \ldots, i_{m-1}}$
for arbitrarily large $m > n$, which is a contradiction. 
Suppose that $\mathcal A$ is $1$-LCC.
Let $\beta \co D^2 \to \R^3$ be a smooth embedding such that $\beta (\del D^2)$ is a meridian of $C_1$.
Cover $\beta( D^2 ) \cap \mathcal A$
by open sets $\{ U_{\gamma} \}_{\gamma \in \Gamma}$ around each of its points, then 
there is a covering $\{ V_{\gamma} \}_{\gamma \in \Gamma}$ such that for all $\gamma \in \Gamma$
we have 
$V_{\gamma} \subset U_{\gamma}$  and
each map $\del D^2 \to (\R^3 - \mathcal A) \cap V_{\gamma}$
can be extended to a map $D^2 \to (\R^3 - \mathcal A) \cap U_{\gamma}$.
We can also suppose that 
$\cup_{\gamma} U_{\gamma}$ is disjoint from $\beta(\del D^2)$.
By Lebesgue lemma there is a refinement of $D^2$ into finitely many small disks with disjoint interiors
such that each of their boundary circles is mapped by $\beta$ into some $V_{\gamma}$.
After a small perturbation we can suppose that each of the $\beta$-images of 
these boundary circles is disjoint from  $C_n$ for some common large $n$ but it is still in 
some $V_{\gamma}$.
Now change $\beta$ on each of the small disks to get a map into $(\R^3 - \mathcal A) \cap U_{\gamma}$.
By Dehn's lemma there are embeddings as well  of the small disks  into $(\R^3 - \mathcal A) \cap U_{\gamma}$.
In this way we get an embedding of the original disk $D^2$ which is disjoint from $\mathcal A$.
This contradicts to 
the fact that every embedded disk $D^2 \subset \R^3$ with boundary circle being a meridian of $C_1$  intersects $\mathcal A$.
\end{proof}

The standard Cantor set $C \subset \R\x \{ 0 \} \x \{ 0 \} \subset \R^3$ is $1$-LCC, because
having a small loop in its complement $\R^3 - C$ yields by approximation a small smooth loop in $\R^3 - C$
 transversal to and disjoint from $\R \x \{ 0 \} \x \{ 0 \}$.
Then deform this loop by compressing it in  a direction  parallel to $\R\x \{ 0 \} \x \{ 0 \}$
until the loop sits in the plane $\{ x \} \x \R^2$ for some number $x \in \R^3 - C$.
After these the loop can be squeezed easily inside this plane to a point in $\R^3 - C$.
This implies that  Antoine's necklace is  a wild Cantor set in $\R^3$.

\subsection{Bing decomposition}

If in the construction of Antoine's necklace
there are always two tori components of $C_{n+1}$ 
in each component of $C_n$, then
we call the arising decomposition Bing decomposition.
Apriori there could be many different Bing decompositions depending on
how the solid tori are embedded into each other.
It is not obvious that we can arrange the components of $C_n$ 
embedded in such a way that $\cap_{n = 1}^{\infty} C_n$
is a Cantor set, which would follow if the maximal diameter of the tori in $C_n$ converges to $0$.
A random
defining sequence can be seen in Figure~\ref{bing_random}.

\begin{figure}[h!]
\begin{center}
\epsfig{file=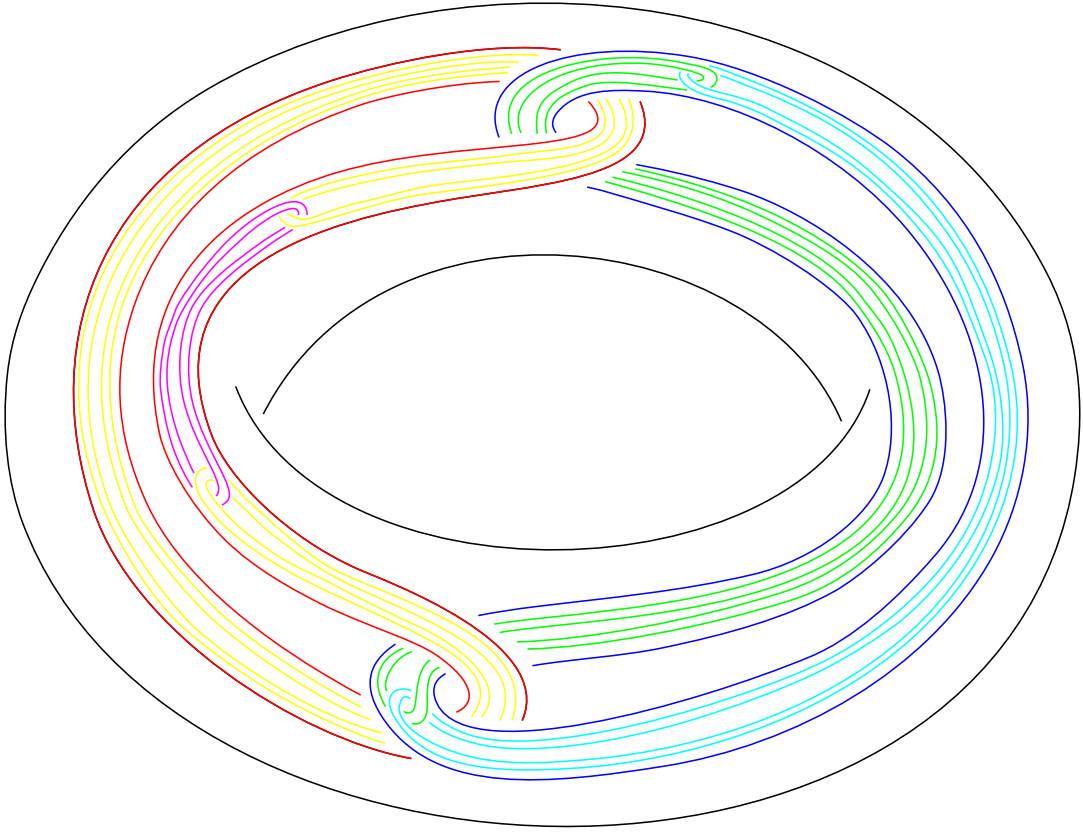, height=11cm}
\end{center} 
\caption{A sketch of a defining sequence of the Bing decomposition. 
We can see the torus $C_1$, the two torus components of $C_2$ and the four torus components of $C_3$.
The maximal diameter of tori in $C_n$ does not converge
to $0$ necessarily.}
\label{bing_random}
\end{figure}

Now we construct a defining sequence, where the maximal diameter of the tori in $C_n$ converges to $0$.
For this, consider the following way to define a finite sequence of finite sequences of embeddings:
$$D_{0},$$
$$D_{0} \supset D_{2,1},$$
$$D_{0} \supset D_{3,1} \supset D_{3,2},$$
$$D_{0} \supset D_{4, 1} \supset \cdots \supset D_{4,3},$$
\begin{figure}[h!]
\begin{center}
\epsfig{file=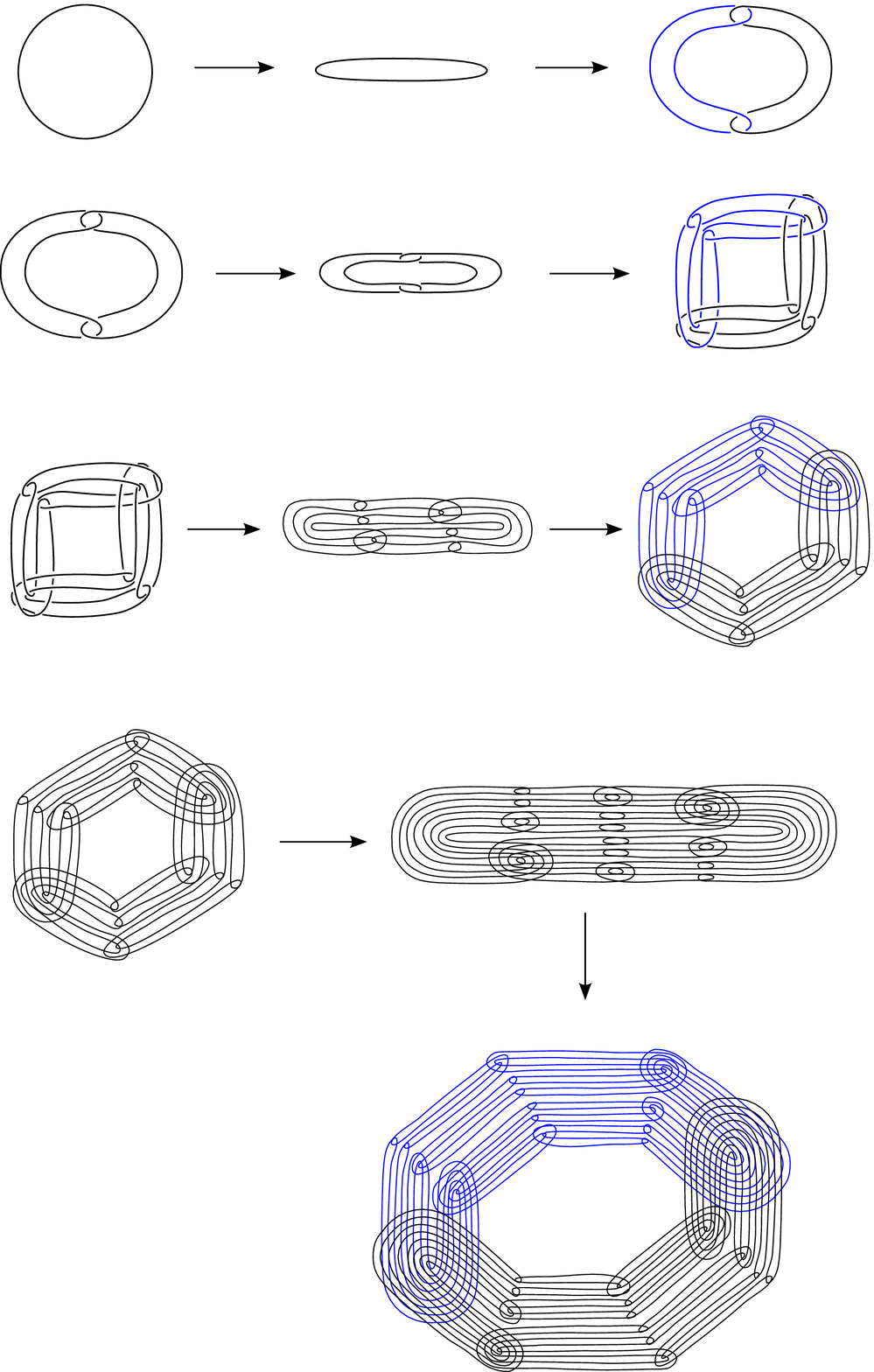, height=17.1cm}
\put(-12.5, 16.1){$D_{1}$}
\put(0, 16.1){$D_{2,1}$}
\put(0, 13.8){$D_{3,2}$}
\put(0.5, 10.6){$D_{4,3}$}
\put(-0.5, 2.1){$D_{5,4}$}
\end{center} 
\caption{A sketch  of constructing the tori $D_{n+1, n}$.
 Instead of the solid tori we just draw their center circles.
We always take the previously obtained  linked tori $D_{n, n-1}$, 
squeeze them to become ``flat''  as the figure shows,
then curve them a little 
 and link them with another copy at the two ``endings''. Hence we get $D_{n+1, n}$. 
 The sequence of embeddings 
 $D_{0} \supset D_{n+1, 1} \supset \cdots \supset D_{n+1,n}$
 can be kept in sight by checking all the smaller linkings.}
\label{bing_deform}
\end{figure}
\ldots, etc., where $D_{n, 0} = D_{0}$ is a solid torus,
$D_{n,k}$ is a disjoint union of $2^{k}$ copies of 
solid tori and
the components of $D_{n,k}$ are pairwise embedded in the components of 
$D_{n,k-1}$, moreover 
these pairs are linked just like in the defining sequence of Bing decomposition, for further
subtleties see Figure~\ref{bing_deform}. 

Arriving to the tori $D_{n+1,n}$ 
and assuming that their meridional size is small enough,
we obtain a regular $2n$-gon-like arrangement of $2^{n}$ copies of solid tori
as Figure~\ref{bing_deform} shows.
Two conditions are satisfied: the meridional size of all the tori is small 
and
an ``edge'' of this $2n$-gon is also small. This means that if this 
$D_{n+1,n}$ is embedded into a torus (as the figure suggests) whose meridional size is small, then the
maximal diameter of the torus components of
 $D_{n+1,n}$ is small if $n$ is large.

\begin{prop}
There is a defining sequence $C_1, \ldots, C_n, \ldots$ of the Bing decomposition, where
the maximal diameter of tori in $C_n$ converges to $0$.
Hence $\cap_{n=1}^{\infty} C_n$ is homeomorphic to the Cantor set.
\end{prop}
\begin{proof}
Let $\ep_n > 0$ be a sequence whose limit is $0$.
Let $n_1$ be so large  that in $C_{n_1}$ in a defining sequence  the meridional size of tori is smaller than $\ep_1$.
Let $m_1$ be so large  that we can embed 
 $D_{m_1+1,m_1}$ into the torus components of $C_{n_1}$ so
 that the maximal diameter of tori in the obtained $C_{n_1 + m_1}$
 is smaller then $\ep_1$.
 Then let $n_2 > n_1 + m_1$ be so large that 
in a continuation of the defining sequence in $C_{n_2}$ the meridional size of tori is smaller than $\ep_2$.
Let $m_2$ be so large  that we can embed 
 $D_{m_2+1,m_2}$ into the torus components of $C_{n_2}$ so
 that the maximal diameter of tori in the obtained $C_{n_2 + m_2}$
 is smaller then $\ep_2$. And so on.
 It is easy to see that the maximal diameter of tori converges to $0$.
 \end{proof}

This implies that the decomposition space of this decomposition is $\R^3$.
For an arbitrary defining sequence 
the space $\cap_{n=1}^{\infty} C_n$ may be not the Cantor set, however the decomposition space 
could be still homeomorphic to the ambient space $\R^3$.
It is a very important observation that the embedding of the tori in $D_{n+1,n}$ 
can be obtained by an isotopy of  $C_1 \subset  \cdots \subset C_{n+1}$
in any defining sequence, see \cite{Bi52}.
By such an isotopy 
for a given defining sequence
we can manage something similar to the previous statement: 
 if $n$ is large enough, then the meridional size of the torus components in $C_n$ is smaller than a given $\ep > 0$.
Then  apply the required isotopy for $C_{n+1}, \ldots, C_{n+k}$ for some large $k$ to make
 the maximal diameter of the torus components of $C_{n+k}$ smaller than $\ep$.
 Note that since $n$ is large enough and all the isotopy happens inside $C_n$, 
 all the isotopy happens  inside an arbitrarily small nbhd of $\cap_{n=1}^{\infty} C_n$.
 This means that for every $\ep > 0$ there is
 a self-homeomorphism $h$ of $\R^3$ with  support $C_1$ such that 
 $h( \mathcal D ) < \ep$
  for every decomposition element $\mathcal D \subset \cap_{n=1}^{\infty} C_n$
  and also 
  $\pi \circ h( \mathcal D)$ stays in the $\ep$-nbhd
  of $\pi ( \mathcal D)$ for some metric on the decomposition space. 
  This condition is called shrinkability criterion
  and it implies that 
  the decomposition space is homeomorphic to the ambient space $\R^3$
  as we will see in the next section.

\bigskip{\Large{\section{Shrinking}}}
\bigskip\large

Let $X$ be a topological space and $\mathcal D$ a decomposition of $X$.
An open cover $\mathcal U$ of $X$ is called $\mathcal D$-saturated if every $U \in \mathcal U$
 is a union of decomposition elements.

\begin{defn}[Bing shrinkability criterion]
Let $\mathcal D$ be a usc decomposition of the space $X$.
We say $\mathcal D$ is \emph{shrinkable}
if
for every open cover $\mathcal V$ and $\mathcal D$-saturated open cover 
$\mathcal U$ there is a self-homeomorphism $h$ of $X$ such that 
for every $D \in \mathcal D$ the set 
$h(D)$ is in some $V \in \mathcal V$
and for every $x \in X$ there is a $U \in \mathcal U$
such that $x, h(x) \in U$. In other words, $h$ shrinks the elements of $\mathcal D$ to arbitrarily small sets and $h$ is $\mathcal U$-close
to the identity. 
We say $\mathcal D$ is \emph{strongly shrinkable}
if
for every open set $W$ containing all the non-degenerate elements of $\mathcal D$
the decomposition $\mathcal D$  is shrinkable so that the support of
$h$ is in $W$.
\end{defn}
In other words $\mathcal D$ is shrinkable if 
its elements can be made small enough simultaneously so that
 this shrinking process does not move the points of $X$ too far in the sense of  measuring the distance in the decomposition space.
If $X$ has a shrinkable decomposition, then we expect that 
the local structure of $X$ is similar to the structure of the nbhds of the decomposition elements. 
 
\begin{prop}
Let $X$ be a regular space and let $\mathcal D$ be a shrinkable usc decomposition of $X$.
If every $x \in X$ has arbitrarily small nbhds satisfying a fixed topological property, then 
every $D \in  \mathcal D$ has arbitrarily small nbhds satisfying the same property.
\end{prop}
\begin{proof}
Let $W$ be an arbitrary nbhd of an element  $D \in \mathcal D$.
Then there is a saturated nbhd $\tilde U_1$ of $D$ such that $\tilde U_1 \subset W$. Let $U_1$ denote  $\pi(\tilde U_1)$.
Since $X_{\mathcal D}$ is regular,
there are open sets $U_2$ and $U_3$ such that
$$\pi(D) \subset U_3 \subset \cl U_3 \subset U_2 \subset \cl U_2 \subset U_1.$$
Then take the sets 
$$
\pi^{-1}(U_3),\mbox{\ }  \pi^{-1}( U_2) - D,\mbox{\ } \pi^{-1}( U_1 - \cl U_3 ),\mbox{\ and\  } X - \pi^{-1}(\cl U_2 ),$$
see Figure~\ref{localprop}.
These yield a $\mathcal D$-saturated open cover $\mathcal U$ of $X$.
Let $\mathcal V$ be an open cover of $X$ which refines $\mathcal U$
and consists of open sets with our fixed property.
Since $\mathcal D$ is shrinkable, we have a homeomorphism
$h \co X \to X$ such that 
$h(D) \subset V$ for some $V \in \mathcal V$ and
$h$ is $\mathcal U$-close to the identity.
Then $D \subset h^{-1}(V)$ so 
it is enough to show that 
$$h^{-1}(V) \subset W.$$
Suppose there exists
some $x \in h^{-1}(V) - W$, then 
$x \in h^{-1}(V) - \pi^{-1}(U_1)$ since $\pi^{-1}(U_1) \subset W$.
Hence among the sets in $\mathcal U$ only $X - \pi^{-1}(\cl U_2 )$
contains $x$ so $h(x) \in V$ has to be 
in $X - \pi^{-1}(\cl U_2 )$ as well.
This implies that 
$V \subset X - \pi^{-1} ( \cl U_3 )$
because $V$ is a subset of some sets in $\mathcal U$.
Also we know that $h(D) \subset \pi^{-1}(U_3)$ since
$D \subset \pi^{-1}(U_3)$.
This means that $h(D) \subset V \subset X - \pi^{-1} ( \cl U_3 )$ cannot hold so
$h^{-1}(V) \subset W$.
\end{proof}
\begin{figure}[h!]
\begin{center}
\epsfig{file=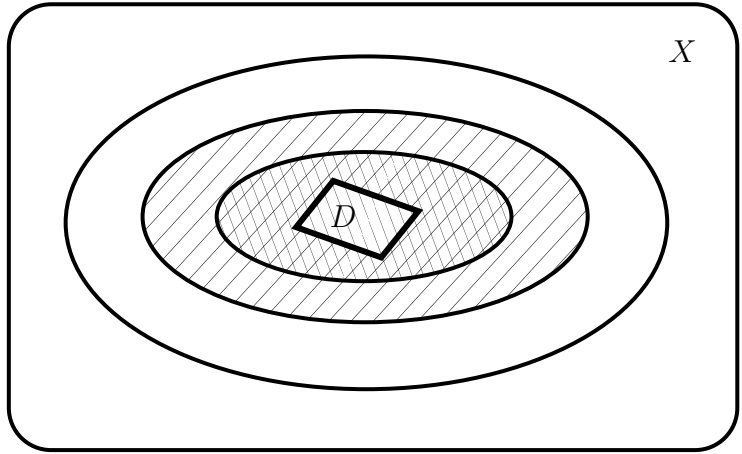, height=6cm}
\put(-1, 5.2){$X$}
\put(-5.5, 3){$D$}
\end{center} 
\caption{The set $D$ and the $\pi$-preimages of the sets $U_3 \subset U_2 \subset U_1$ in $X$.
In the figure the sets $\pi^{-1}(U_3)$ and $\pi^{-1}( U_2) - D$ are shaded.
}
\label{localprop}
\end{figure}

For example every decomposition element of a shrinkable decomposition of a manifold is cellular.

It is often not too difficult to check whether a decomposition of a space $X$ is shrinkable. 
A corollary of shrinkability is
that the decomposition space is homeomorphic to $X$.
This is often applied when we want to construct embedded manifolds 
 and the construction uses mismatched pieces, which we eliminate by 
 taking them as the decomposition elements
 and then looking at the decomposition space.

\begin{defn}[Near-homeomorphism, approximating by homeomorphism]
Let $X$ and $Y$ be topological spaces.
An $f \co X \to Y$ surjective map is a \emph{near-homeomorphism}
if
for every open covering  $\mathcal W$ 
of $Y$ there is a homeomorphism $h \co X \to Y$ such that
for every $x \in X$ the points $f(x)$ and $h(x)$ are in some $W \in \mathcal W$,
 in other words $h$ is \emph{$\mathcal W$-close} to $f$.
\end{defn}

If $(Y, \varrho)$ is a metric space, then $f$ being  a near-homeomorphism implies
that $f$ can be approximated by homeomorphisms in the possibly
infinite-valued metric $d(f, g) = \sup \{ \varrho(f(x), g(x)) \}$.
Notice that if $f \co X \to Y$ is a near-homeomorphism, then 
$X$ and $Y$ are actually homeomorphic.

The main result is that a
usc decomposition  yields a homeomorphic decomposition space if
the decomposition is shrinkable.
This is applied in major $4$-dimensional results:
in the disk embedding theorem and in the proof of the $4$-dimensional topological 
Poincar\'e conjecture \cite{Fr82, BKKPR}.
It is extensively applied in constructing  approximations of manifold embeddings in dimension $\geq 5$, see
\cite{AC79} and Edwards's cell-like approximation theorem.

For an open cover $\mathcal W$ of a space $X$ and a subset $A \subset X$
let $St(A, \mathcal W)$ denote the subset 
$$
\bigcup \{ W \in \mathcal W :  W \cap A \neq \emptyset \}.
$$
This is called the \emph{star} of $A$ and it is a nbhd of $A$.
Of course
if $A \subset B$, then $St(A, \mathcal W) \subset St(B, \mathcal W)$.
If
$\mathcal W'$ is an open cover which is a refinement of the open cover $\mathcal W$,
then obviously
$St(A, \mathcal W') \subset St(A, \mathcal W)$.
We will use often that if the covering $\mathcal W'$ is a \emph{star-refinement}
of the covering $\mathcal W$, that is
the collection 
$$\left\{ St(W_{\al}, \mathcal W') : W_{\al} \in \mathcal W' \right\}$$
of stars of elements of $\mathcal W'$
is a refinement of $\mathcal W$, then 
for every point $x \in X$ we have
$
St(\{ x \}, \mathcal W' ) \subset W$ for some $W \in \mathcal W$.

The following theorem requires a complete metric on the space $X$, for example
the statement holds for an arbitrary manifold.

\begin{thm}\label{shrinkinglocallycompact}
Let $\mathcal D$ be a usc decomposition of a  space $X$ admitting a complete metric. 
Then  
the following are equivalent:
\begin{enumerate}[\rm (1)]
\item
the decomposition map $\pi \co X \to X_{\mathcal D}$
is a near-homeomorphism,
\item
$\mathcal D$ is shrinkable.
\end{enumerate}
If additionally $X$ is also locally compact and separable, then shrinkability is equivalent to
\begin{enumerate}[\rm (3)]
\item
if $C \subset X$ is an arbitrary compact set, $\ep > 0$ and $\mathcal U$ is
a $\mathcal D$-saturated open cover of $X$,
then there is a homeomorphism $h \co X \to X$ such that 
${\mathrm{diam}}\thinspace h(D) < \ep$ for every $D \subset C$, $D \in  \mathcal D$
and $h$ is $\mathcal U$-close to the identity.
\end{enumerate}
\end{thm}
\begin{proof}
{\textbf{Near-homeomorphism (1) implies shrinking (2) and (3).}}
Of course (2) implies (3) so we are going to prove only that 
(1) implies (2).
At first, suppose that 
the decomposition map $\pi \co X \to X_{\mathcal D}$
is a near-homeomorphism. We have to show that $\mathcal D$ is shrinkable by finding 
an appropriate homeomorphism $h$.
We know that since $X$ is metric, the decomposition space $X_{\mathcal D}$ is metrizable hence it is 
paracompact. 
(To show that  $\mathcal D$ is shrinkable, we will  use only  that the space $X$ is paracompact
and $T_4$.)
Let $\mathcal V$ be an open cover  and let 
$\mathcal U$ be a 
$\mathcal D$-saturated open cover of $X$.
Take the open covering $\{ \pi ( U) : U \in \mathcal U \}$ of $X_{\mathcal D}$.
Since $X_{\mathcal D}$ is paracompact, this covering has a star-refinement 
$\mathcal W_0$, i.e.\ $\mathcal W_0$ is a covering 
and the collection of stars of elements of $\mathcal W_0$, that is
the collection  
$$\left\{ St(W_{\al}, \mathcal W_0) : W_{\al} \in \mathcal W_0 \right\}$$
is a refinement of $\{ \pi ( U) : U \in \mathcal U \}$, see \cite[Section~8.3]{Du66}.
Similarly $\mathcal W_0$ has a star-refinement covering $\mathcal W_1$.
Then there is a homeomorphism 
$$
h_1 \co X \to X_{\mathcal D}
$$
which is $\mathcal W_1$-close to $\pi$ because $\pi$ is a near-homeomorphism.
Take the open cover 
$$
\mathcal W_1 \bigcap h_1(\mathcal V) = \{ W \cap h_1(V) : W \in \mathcal W_1, V \in \mathcal V \}
$$
and a star-refinement $\mathcal W_2$ of it. Of course 
$\mathcal W_2$ is a star-refinement of
$\mathcal W_1$ and $h_1(\mathcal V)$ as well.
There is a homeomorphism 
$$
h_2 \co X \to X_{\mathcal D}
$$
which is  $\mathcal W_2$-close to $\pi$.
Let $h \co X \to X$ be the composition
$$
h_1^{-1} \circ h_2.
$$
At first we show that 
$h$ shrinks every decomposition element $D \in \mathcal D$ 
into some $V \in \mathcal V$.
Let $D \in \mathcal D$. 
It is enough to show that 
$h_2 (D) \subset h_1(V)$ for some $V \in \mathcal V$.
We have that for every $x \in D$ the points 
$\pi(D)$ and $h_2(x)$ are in the same  $W_x \in \mathcal W_2$
so 
$$
h_2(D) \subset St( \{ \pi(D) \}, \mathcal W_2 )  \subset h_1(V)
$$
for some $V \in \mathcal V$ because $\mathcal W_2$ is a star-refinement of
 $h_1(\mathcal V)$.
Now we show that $h$ is $\mathcal U$-close to the identity.
 We  have that for every $x \in D$ the points 
$\pi(D)$ and $h_1(x)$ are in the same  $W_x \in \mathcal W_1$ because $h_1$ is $\mathcal W_1$-close to $\pi$ so
$$
h_1(D) \subset St( \{ \pi(D) \}, \mathcal W_1 ).$$
Since $\mathcal W_2$ is a refinement of  $\mathcal W_1$,
we have 
 $$
 h_2(D) \subset St( \{ \pi(D) \}, \mathcal W_2 ) \subset St( \{ \pi(D) \}, \mathcal W_1 ).$$
These imply that 
$$
h_1(D) \cup h_2(D) \subset St( \{ \pi(D) \}, \mathcal W_1 ) \subset W_0
$$
for some 
$W_0 \in \mathcal W_0$ because $\mathcal W_1$ is a star-refinement of
 $\mathcal W_0$.
Hence for every $D \in \mathcal D$ we have
$$
D \cup h(D) = h_1^{-1} \circ h_1 ( D \cup h(D) )  = 
h_1^{-1} ( h_1 ( D ) \cup h_2 ( D ) )  \subset h_1^{-1} ( W_0 )
$$
so if we show that $$h_1^{-1} ( W_0 ) \subset U$$ for some $U \in \mathcal U$, then we 
prove the statement.
Since $h_1$ and $\pi$ are $\mathcal W_1$-close, they are 
$\mathcal W_0$-close as well.
This means that
for every $x \in X$ 
the points $\pi(x)$ and $h_1(x)$ are in the same $W_x \in \mathcal W_0$. 
So
if $x \in h^{-1}_1 ( W_0 )$, then
$$
\pi(x) \in St( W_0, \mathcal W_0 ),$$
which gives that 
$$
\pi( h^{-1}_1 ( W_0 ) ) \subset  St( W_0, \mathcal W_0 ) \subset \pi(U)
$$
for some $U \in \mathcal U$
because
$W_0$ is a star-refinement of $\pi(\mathcal U)$.
Then the statement follows because 
$$
 h^{-1}_1 ( W_0 ) \subset \pi^{-1} \circ \pi  (h^{-1}_1 ( W_0 )  ) \subset \pi^{-1} \circ \pi (U) = U.
$$

{\textbf{Shrinking (2) or (3)  implies near-homeomorphism (1).}}
At first observe that in the case of (3) 
 if $X$ is locally compact and separable, then $X$ is $\sigma$-compact so $X$ is the
  union $\cup_{n=1}^{\infty} C_n$ of countably many compact sets 
  $$C_1 \subset C_2 \subset \cdots \subset C_n \subset \cdots.$$
  We also suppose that 
  every $C_n$ is $\mathcal D$-saturated and has non-empty interior. 
 Let $\mathcal W$ be an arbitrary open cover
of $X_{\mathcal D}$.
We have to construct a homeomorphism $h \co X \to X_{\mathcal D}$
which is $\mathcal W$-close to $\pi$.
 At first, we construct a sequence 
$$
\mathcal U_0, \mathcal U_1, \ldots, \mathcal U_n, \ldots
$$
of $\mathcal D$-saturated open covers of $X$
and a sequence
$$
h_0, h_1, \ldots, h_n, \ldots
$$
of self-homeomorphisms of $X$
with some useful properties.
Let $\mathcal U_0$ be 
a $\mathcal D$-saturated open cover of $X$ such that
the collection of the closures of the elements of $\mathcal U_0$ refines the open cover $\pi^{-1}(\mathcal W)$.
This obviously exists because
$X_{\mathcal D}$ is regular so around every point of $X_{\mathcal D}$
there is a small closed nbhd contained in some element of $\mathcal W$.
Let $h_0$ be the identity homeomorphism.
Let $\ep_n > 0$ be a decreasing sequence converging to $0$.
Define $\ep_0$ to be $\infty$. Denote the metric on $X$ by $d$.
Suppose inductively that 
we constructed already the covers 
$
\mathcal U_0, \ldots, \mathcal U_n
$
and the homeomorphisms
$
h_0,  \ldots, h_n
$
with the following properties:
\begin{enumerate}\label{egyesketto}
\item
\begin{enumerate}
\item
$\mathcal U_{i+1}$ is a $\mathcal D$-saturated open cover, which refines $\mathcal U_{i}$ for $0 \leq i \leq n-1$,
\item
for all $0 \leq i \leq n$
the set $\mathcal U_i$ refines the collection of $\ep_i$-nbhds of the elements of $\mathcal D$
and also refines the collection $\{ \pi^{-1}(B_{\ep_i}(y)) : y \in X_{\mathcal D} \}$, where
$B_{\ep_i}(y)$ is the open ball of radius $\ep_i$ around $y$,
\end{enumerate}
\item
\begin{enumerate}
\item
for every $0 \leq i \leq n-1$ 
every $D \in \mathcal D$ has a nbhd $U \in \mathcal U_i$ such that 
for every $U' \in \mathcal U_{i+1}$ which contains
$D$ we have
$$
h_i(U' ) \cup h_{i+1}(U') \subset h_i(U),$$
\item[(b)]
for every $0 \leq i \leq n$ 
the diameter of each $h_i(U)$,  $U \in \mathcal U_i$, is smaller than $\ep_i$,
\item[(b$'$)]
in the case of $X$ is $\sigma$-compact we require only that 
for every $0 \leq i \leq n$ and for  every nbhd $U \in \mathcal U_i$ such that $U\cap C_i \neq \emptyset$
the diameter of each $h_i(U)$ is smaller than $\ep_i$.
\end{enumerate}
\end{enumerate}
There will be some important corollaries of these constructions. 
Part (a) of (2) implies that 
 every $D \in \mathcal D$ has a nbhd $U \in \mathcal U_{i}$
such that  
for every $k \geq 1$ and $U' \in \mathcal U_{i+k}$ which contains
$D$ we have
\begin{equation}\label{tart}
h_{i+k}(U') \subset h_i(U).
\end{equation}
For $k=1$ this is immediate from (2)(a) and for $k \geq 2$ this follows by a simple induction.
This means that 
once we have $U_n$ and $h_n$ for every $n \in \N$ satisfying
(1) and (2), the
sequence $h_n$ is a Cauchy sequence in the sense of  
local uniform convergence 
in the space of maps of $X$ into $X$.
Indeed, if $x \in X$, then 
for some $D \in \mathcal D$ we have $x \in D$ and then 
$D$ has a nbhd $U \in \mathcal U_{n}$ for every $n$ 
such that  by applying (\ref{tart})
for all $k \in \N$ 
\begin{equation}\label{tart2}
h_{n+k}(D) \subset h_n(U),
\end{equation}
which means that 
$d( h_n(x), h_{n+k}(x) ) < \ep_n$ for all $x \in X$ 
by (2)(b).
 In the case of $X$ is $\sigma$-compact we have that
 for some $m \in \N$  the intersection $D \cap C_n \neq \emptyset$ for $n \geq m$
 hence by (2)(b$'$)  we have  
  $\mathrm{diam}\thinspace h_n(U) < \ep_n$
  for every $n \geq m$ and nbhd $U \in \mathcal U_{n}$ of $D$.
 This implies that for all  $n \geq m$ we get  
 $d( h_n(x), h_{n+k}(x) ) < \ep_n$ for all $k$ and $x \in  D$, where $D \subset C_n$. 
Since $(X, d)$ is complete, the sequence
$h_n$ converges locally uniformly to a continuous map $$\chi \co X \to X,$$ which will be a good candidate
for obtaining our desired  near-homeomorphism.

{\textbf{Defining  $\mathcal U_{n+1}$ and $h_{n+1}$.}}
So let us return to the definition of the covers $\mathcal U_n$ and homeomorphisms $h_n$.
Suppose inductively that 
we constructed already the covers 
$
\mathcal U_0, \ldots, \mathcal U_n
$
and the homeomorphisms
$
h_0,  \ldots, h_n
$
with the  properties (1) and (2). We are going  to define $\mathcal U_{n+1}$ and $h_{n+1}$.
The metrizable space $X_{\mathcal D}$ is paracompact
so the open cover $\pi ( \mathcal U_n )$
has a star-refinement whose $\pi$-preimage $\mathcal U_n'$ is a $\mathcal D$-saturated open cover of $X$, which
star-refines $\mathcal U_n$. 
Let $\mathcal V$ be an open cover of $X$ such that the diameter of each of its elements is smaller than $\ep_{n+1}$.
Then we have two possibilities.
\begin{itemize}
\item
If $\mathcal D$ is shrinkable, then 
there is a self-homeomorphism $H$ of $X$, which is
$h_n(  \mathcal U_n' )$-close to the identity and 
shrinks the elements of 
$h_n(\mathcal D)$ into the sets of $\mathcal V$. Let 
$$
h_{n+1} = H \circ h_n.$$
Clearly the diameter of each
$h_{n+1} (D)$, where $D \in \mathcal D$, is smaller than $\ep_{n+1}$.
\item
If only those elements of $\mathcal D$ are shrinkable which are in  a chosen compact set
as we suppose in (3) of the statement of Theorem~\ref{shrinkinglocallycompact}, then
there is a homeomorphism
$H_0 \co X \to X$ such that
the elements of $\mathcal D$ in the compact set $C_{n+1}$
are mapped by $H_0$ into some element of $h_n^{-1}(\mathcal V)$
and $H_0$ is $\mathcal U_n'$-close to the identity.
This implies that $h_n \circ H_0 \circ h_n^{-1}$ is such a self-homeomorphism of $X$ that 
maps the elements of $h_n(\mathcal D)$ which are in $h_n(C_{n+1})$ into the sets of $\mathcal V$ and
it is $h_n(  \mathcal U_n' )$-close to the identity. Denote $h_n \circ H_0 \circ h_n^{-1}$ by $H$.
Then 
let $$
h_{n+1} = H \circ h_n.$$
So $h_{n+1}$ maps
every $D \in \mathcal D$, $D \subset C_{n+1}$
into a set of diameter smaller than 
$\ep_{n+1}$.
\end{itemize}
The definition of $\mathcal U_{n+1}$ is a little more complicated.
For every $U_n' \in \mathcal U_n'$
$$
h_{n+1}(U_n') \subset h_n(St(U_n', \mathcal U_n' ))$$
because
$$
h_{n+1}(U_n') = H \circ h_{n}(U_n') \subset St( h_n (U_n'), h_n ( \mathcal U_n' ) )
$$
since $H$ is $h_n(  \mathcal U_n' )$-close to the identity
and also
$$
St( h_n (U_n'), h_n ( \mathcal U_n' ) )
= h_n(St(U_n', \mathcal U_n' )).$$
The covering $\mathcal U_n'$ star-refines $\mathcal U_n$
so for every $U_n' \in \mathcal U_n'$ there is an $U_n \in \mathcal U_n$ such that
$$
h_n(St(U_n', \mathcal U_n' )) \subset h_n(U_n),$$
which obviously implies that 
for every $U_n' \in \mathcal U_n'$ there is an $U_n \in \mathcal U_n$ such that
$$
h_n (U_n') \cup  h_{n+1}(U_n') \subset h_n(St(U_n', \mathcal U_n' )) \subset h_n(U_n).$$
Let $\mathcal S$ be a $\mathcal D$-saturated  open cover of $X$
with the following properties:
\begin{enumerate}[\rm (i)]
\item
the elements of $\mathcal S$ are nbhds of the elements of $\mathcal D$
such  that the diameter of each 
$h_{n+1} (S)$, where $S \in \mathcal S$, is smaller than $\ep_{n+1}$ (in the case of $S \cap C_{n+1} \neq \emptyset$ 
if  $X$ is $\sigma$-compact),
\item
$\mathcal S$ refines the collection of $\ep_{n+1}$-nbhds of the elements of $\mathcal D$,
\item
$\mathcal S$ also refines the $\mathcal D$-saturated  coverings
\begin{enumerate}
\item
$\mathcal U_n'$ and
\item
the collection $\{ \pi^{-1}(B_{\ep_{n+1}}(y)) : y \in X_{\mathcal D} \}$,
\end{enumerate}
\item
for every $S \in \mathcal S$ there is a $U_n \in \mathcal U_n$ such that
$$
h_n (S) \cup  h_{n+1}(S)  \subset h_n(U_n).$$
\end{enumerate}
Let $\mathcal U_{n+1}$ be the $\pi$-preimage of an open cover of $X_{\mathcal D}$ which 
star-refines the open cover $\pi(\mathcal S)$. It follows that 
$\mathcal U_{n+1}$  star-refines $\mathcal S$. 
After we defined $\mathcal U_{n+1}$ and $h_{n+1}$
let us check if 
$
\mathcal U_0, \ldots, \mathcal U_{n+1}
$
and 
$
h_0,  \ldots, h_{n+1}
$
satisfy the conditions (1) and (2) on page~\pageref{egyesketto}. 
The cover $\mathcal U_{n+1}$ refines the cover $\mathcal U_{n}$
because $\mathcal U_{n}'$ refines $\mathcal U_{n}$, $\mathcal S$ refines  $\mathcal U_{n}'$ by 
(iii)(a) and $\mathcal U_{n+1}$ refines $\mathcal S$. So (1)(a) holds.
Also (1)(b) holds because of (ii) and (iii)(b).
To prove (2)(a) observe that 
for every $D \in \mathcal D$ the set $St(D, \mathcal U_{n+1})$ 
is a subset of $St(U, \mathcal U_{n+1})$ for some 
$U \in \mathcal U_{n+1}$. Then  $St(D, \mathcal U_{n+1}) \subset S$ for some
$S \in \mathcal S$ since $\mathcal U_{n+1}$  star-refines $\mathcal S$.
By (iv) there exists 
a $U \in \mathcal U_n$ such that
$$
h_n (S) \cup  h_{n+1}(S)  \subset h_n(U)$$
so
$$
h_n (St(D, \mathcal U_{n+1}) ) \cup  h_{n+1}( St(D, \mathcal U_{n+1}) )  \subset h_n(U).$$
But every $U' \in \mathcal U_{n+1}$
which contains 
$D$ is in $St(D, \mathcal U_{n+1})$ so we have
$$
h_n (U') \cup  h_{n+1}(U')  \subset h_n(U),$$
which proves (2)(a).
Finally,
the diameter of each $h_{n+1}(U)$,  $U \in \mathcal U_{n+1}$, is smaller than $\ep_{n+1}$ 
(if $U \cap C_{n+1} \neq \emptyset$ 
in the case of $\sigma$-compact $X$),
because
$\mathcal U_{n+1}$ refines $\mathcal S$ and we can apply (i).

{\textbf{Constructing the near-homeomorphism.}}
After having these infinitely many $\mathcal D$-saturated open coverings
$$
\mathcal U_0, \mathcal U_{1}, \ldots
$$
and homeomorphisms
$$
h_0, h_{1}, \ldots
$$
take the map 
$$
\chi \co X \to X$$
that we obtained applying (\ref{tart2}) and defined to be the pointwise limit of the sequence $h_n$.  
At first, we show that $\chi$ is surjective.
Let $x \in X$ and $x_n = h_{n+1}^{-1}(x)$.
Let $D_n \in \mathcal D$ be such that $x_n \in D_n$, then by (2)(a) we get a nbhd $U_n \in \mathcal U_n$ of $D_n$
such that 
for every $U' \in \mathcal U_{n+1}$ containing $D_n$ we have
$$
h_n(U' ) \cup h_{n+1}(U') \subset h_n(U_n).$$
In this way we get a decreasing sequence 
$$
U_1 \supset U_2 \supset \cdots \supset U_n \supset \cdots
$$
because of the following.
It is enough to show that $D_n \subset U_{n+1}$ as well, then
by (2)(a) we obtain $h_n(U_{n+1}) \subset h_n(U_n)$ so $U_{n+1} \subset U_n$.
But $D_n \subset U_{n+1}$ because
$$
h_{n+2}(V) \subset h_{n+1}(U_{n+1})
$$
by (2)(a) for every $V \in \mathcal U_{n+2}$
containing $D_{n+1}$, so we also have 
$$
h_{n+2} (x_{n+1}) \in h_{n+2} (D_{n+1})  \subset h_{n+2} ( V),$$
which
implies 
$$
x \in h_{n+1} ( U_{n+1} )
$$
hence
$$
h_{n+1}(x_n) \in h_{n+1}  ( U_{n+1} )\mbox{\ \ \ \ and so\ \ \ \ }x_n \in U_{n+1}$$
but
$U_{n+1}$ is $\mathcal D$-saturated
hence also $D_n \subset U_{n+1}$.

The sequence $(x_n)$  has a   Cauchy (hence convergent) subsequence:
since $x_ n \in U_n$,
for all $k \geq 0$
$$
x_{n+k} \in U_n
$$
and
for every $\ep > 0$ there is $\ep_n < \ep$
such that $U_n$
 is in the $\ep$-nbhd of some $D \in \mathcal D$ by (1)(b). 
Since the metric space $X$ is complete, there is an $x_0 \in X$ such that a subsequence $(x_{n_k})$ of $(x_n)$
converges to $x_0$.

All of these imply that because of the definition of $\chi$ and the locally uniform convergence of $h_n$ we have 
$$
\chi(x_0)  = \lim_{k \to \infty} h_{n_k + 1}( x_{n_k}) = \lim_{k \to \infty} x = x. 
$$
This means $\chi$ is surjective.

It will turn out that $\chi$ is not injective so it is not a homeomorphism. However,  
the composition $\pi \circ \chi^{-1}$ of the relation $\chi^{-1}$ and the decomposition map $\pi$ is a homeomorphism.
To see this, we show that 
the sets $\chi^{-1}(x)$, where $x \in X$, are exactly the decomposition elements of $\mathcal D$.
By  (\ref{tart2})
for every $n \in \N$ and $D \in \mathcal D$ there is a nbhd $U \in \mathcal U_n$ of $D$ 
such that for every $k \geq 0$
$$
h_{n+k} ( D) \subset h_n(U)$$
hence
$$
\chi ( D ) =  \lim_{k \to \infty} h_{n+k}( D ) \subset \cl h_n(U).
$$
It is a fact that 
${\mathrm {diam}}\thinspace \cl A = {\mathrm {diam}}\thinspace A$
 for an arbitrary subset $A$ of a metric space so by (2)(b)  we obtain 
 $\chi(D) < \ep_n$ for each $n$
 and by (2)(b$'$) for some $C_{m} \supset D$
 we obtain 
 $\chi(D) < \ep_n$ for each $n \geq m$, which implies that 
 $\chi(D)$ is a point. To show that 
 the $\chi$-preimage of a point is not bigger than a decomposition element, 
 observe that for different elements $D_1$ and $D_2$ 
and for large enough $n$ by (1)(b) there are $U, V \in \mathcal U_n$ which lie
in the small $\ep_n$-nbhds of
$D_1$ and $D_2$, respectively, 
hence $U$ and $V$ are disjoint.
Then similarly to above, 
$$
\chi ( D_1 ) \subset \cl h_n(U)\mbox{\ \ \ \ and\ \ \ \ }\chi ( D_2 ) \subset \cl h_n(V),
$$
which implies that $\chi ( D_1 )$ and $\chi ( D_2 )$ are different so 
the sets $\chi^{-1}(x)$, where $x \in X$, are exactly the decomposition elements of $\mathcal D$.

This means that $\pi \circ \chi^{-1}$ is a bijection. Its inverse is continuous because 
$\chi$ is continuous and $\pi$ is a closed map since the decomposition is usc.
To prove that $\pi \circ \chi^{-1}$ is continuous 
it is enough to show that 
$\chi$ is a closed map.
Let $A \subset X$ be a closed set and
observe that 
a point $y \in X$ is in $X - \chi (A)$ if and only if
$\chi^{-1}(y) \cap \chi^{-1} ( \chi ( A ) ) = \emptyset$,
which holds exactly if $\chi^{-1}(y) \cap \pi^{-1} ( \pi ( A )) = \emptyset$.
This means that in order to show that $\chi (A)$ is closed
it is enough to prove that 
for any decomposition element $D$ such that 
$D \cap \pi^{-1} ( \pi ( A )) = \emptyset$
the point $\chi (D)$ is an inner point of $X - \chi (A)$.
If $\ep_n$ is small enough, then
since $D \cap \pi^{-1} ( \pi ( A )) = \emptyset$,
by (1)(b)
for every $U_n \in \mathcal U_n$ containing $D$ we have
$$
\cl U_n \cap \cl St(A, \mathcal U_n) = \emptyset.$$
By (\ref{tart2}) we have 
$\chi ( D ) \in h_n ( \cl U_n )$
and
obviously
$$
\chi ( A ) \subset  h_n ( \cl  St(A, \mathcal U_n) ) = \cl h_n (  St(A, \mathcal U_n) ) 
$$
so finally we get
$$
\chi (D) \in X - \cl h_n (  St(A, \mathcal U_n) ) \subset X - \chi ( A )
$$
implying that
$\chi (D)$ is an inner point of $X - \chi (A)$.
As a consequence the map 
$\pi \circ \chi^{-1}$
is a homeomorphism. We have to prove that 
it is $\mathcal W$-close to the identity.
By (\ref{tart}) for every $D$ and for all $n$ 
there exist $U_n \in \mathcal U_n$ nbhds of $D$
such that 
$$
U_0 = h_0 ( U_0 ) \supset h_1 (U_1) \supset \cdots \supset h_n (U_n) \supset \cdots.
$$
So $h_n(D) \in U_0$ for every $n$ and
then 
$\chi(D) \in \cl U_0$.
Since the collection of the closures of the elements in $\mathcal U_0$ refines the cover $\pi^{-1}(\mathcal W)$,
both of $D$ and $\chi(D)$ are in the same $\pi^{-1}(W) \in \mathcal W$.
This implies that if we denote $\chi(D)$ by $x$, then 
both of $\chi^{-1} ( x )$ and $x$ are in  $\pi^{-1}(W)$.
As a result
$\chi ( \chi^{-1} ( x )) = x$ and $\chi(x)$ are in $\chi (\pi^{-1}(W))$
so by applying the map $\pi \circ \chi^{-1}$ we get that 
$$
\pi \circ \chi^{-1} (x)\mbox{\ \ \ \ and\ \ \ \ }\pi \circ \chi^{-1} ( \chi(x) ) = \pi(x)$$
are in $W$. 
This shows that 
$\pi \circ \chi^{-1}$ is $\mathcal W$-close to $\pi$.
\end{proof}


The goal of most of the applications of shrinking 
is
to obtain some kind of embedding of a manifold by the process of approximating a given map.
Let $\R^d_+$ denote the closed halfspace in $\R^d$.

\begin{defn}[Flat subspace and  locally flat embedding]
Let $A \subset X$ be a chosen subspace of a topological space $X$.
We say that the subspace $B \subset X$ homeomorphic to $A$ 
is \emph{flat} if
there is a homeomorphism $h \co X \to X$
such that $h( B ) = A$.
Let $X$ be an $n$-dimensional manifold. An embedding $e \co B \to X$ of a  $d$-dimensional manifold $B$
is \emph{locally flat} if every  point $e(b)$ has a nbhd $U$  in $X$
such that 
the pair 
$$(U, e(B) \cap U)\mbox{\ is homeomorphic to\ }
\left\{
\begin{array}{ccc}
 (\R^n, \R^d) & \mbox{if $b$ is an inner point of $B$}  \\
 (\R^n, \R^d_+) &   \mbox{if $b$ is a boundary point of $B$}. 
\end{array}
\right.
$$
\end{defn}

\begin{defn}[Collared and bicollared subspaces]
The subspace $A \subset X$ is 
\emph{collared} if there is an embedding $f \co A \x [0,1) \to X$ onto an open subspace of $X$
such that $f(a, 0) = a$.
The subspace $A \subset X$ is 
\emph{bicollared} if there is an embedding $f \co A \x (-1,1) \to X$
such that $f(a, 0) = a$.
The subspace $A \subset X$
is \emph{locally collared} (or  \emph{locally bicollared}) if
every $a \in A$ has a nbhd $U$ in $X$ such that $A \cap U$ is
collared (resp.\ bicollared).
\end{defn}

A typical application of  shrinking is the following.
\begin{thm}
Let $X$ be an $n$-dimensional manifold with boundary $\del  X$.
Then $\del X$ is collared in $X$.
\end{thm}
\begin{proof}
Attach the manifold $\del X \x [0, 1]$ to $X$ along $\del X \subset X$ by the identification
$$\varphi \co \del X \x \{ 0 \} \to X,$$
$$\varphi (x, 0) = x.$$
In this way we get a manifold $\tilde X$, which contains
the attached $\del X \x [0, 1]$ as a subset. 
The boundary of $\tilde X$ is $\del X \x \{ 1 \}$ and so the
boundary $\del \tilde X$  is obviously collared.
Let $\mathcal D$ be the decomposition of $\tilde X$ into the intervals $\{ \{ x \} \x  [0,1] : x \in \del X \}$ 
and the singletons in $\tilde X - \del X \x [0, 1]$.
Then $X$ and the quotient space $\tilde X_{\mathcal D}$
are homeomorphic by the map
$$\alpha \co X \to \tilde X_{\mathcal D},$$
$$\alpha(x) = [x],$$
where $[x]$ denotes the equivalence class of $x$.
Indeed, $\alpha$ is a bijection mapping $X - \del X$
to the classes consisting of single points and
mapping the boundary points $x \in \del X$ to
the class $[x]$. It is easy to see that $\alpha$ and also $\alpha^{-1}$ are continuous so $\alpha$
is a homeomorphism.
If we prove that $\tilde X$ is also homeomorphic to the decomposition space $\tilde X_{\mathcal D}$ by a map $\beta$
as the diagram
\begin{center}
\begin{graph}(6,2)
\graphlinecolour{1}\grapharrowtype{2}
\textnode {X}(0.5,1.5){$X$}
\textnode {T}(5.5, 1.5){$\tilde X$}
\textnode {Q}(3, 0){$\tilde X_{\mathcal D}$}
\diredge {X}{Q}[\graphlinecolour{0}]
\diredge {T}{Q}[\graphlinecolour{0}]
\freetext (1.9,1){$\alpha$}
\freetext (4, 1){$\beta$}
\end{graph}
\end{center}
shows, 
 then 
we obtain that $X$ and $\tilde X$ are homeomorphic through the map $\beta^{-1} \circ \alpha$, 
 which finishes the proof.
A homeomorphism $\beta$ exists if
we prove that ${\mathcal D}$
is shrinkable because then $\pi \co \tilde X \to \tilde X_{\mathcal D}$ is a near-homeomorphism.
Let $\mathcal V$ be an arbitrary open cover of $\tilde X$ and
let $\mathcal U$ be a  $\mathcal D$-saturated open cover of $\tilde X$.
Let $\mathcal W$ be a refinement of $\mathcal V$
such that $\mathcal W$ contains all the small nbhds of the form 
$U_{x} \x [1, 1- \ep_{x})$ for all  ${(x,1)} \in \del \tilde X$
and for some appropriate $\ep_{x} > 0$ and relative nbhd $U_{x} \subset \del X$.
We also suppose that in $\mathcal W$
the nbhds of the inner points
of $\tilde X$ are
only these or such nbhds which do not intersect $\del \tilde X$.
We will apply Theorem~\ref{shrinkinglocallycompact}.
Let $C \subset \tilde X$ be a compact set and let $E \subset \tilde X$
be a compact set 
containing  the attached $\{x\} \x [0, 1]$ for all $(x, 1) \in \del \tilde X$
such that
$\{x\} \x [0, 1]$ intersects $C$.
Since $E$ is compact, there are finitely many nbhds in $\mathcal W$ and also in $\mathcal U$
which cover $E$. Let us restrict ourselves to these finitely many nbhds. 
Let $\ep>0$ be such that 
$\ep < \ep_x$ for all these finitely many points $(x, 1) \in \del \tilde X$.
Let $U$ be the union  of the chosen finitely many nbhds in $\mathcal U$
and let $\de > 0$ be such that 
for a metric on $\tilde X$ the $\delta$-nbhd of
$$\bigcup_{(x,1) \in \del \tilde X \cap E} \{x\} \x [0, 1]$$
is inside $U$. 
Then define a homeomorphism 
$h \co \tilde X \to \tilde X$ 
which maps
$$\bigcup_{(x,1) \in \del \tilde X \cap E} \{x\} \x [0, 1]$$
into 
$$\bigcup_{(x,1) \in \del \tilde X \cap E} \{x\} \x [1, 1- \ep)$$ by 
mapping each arc $\{x\} \x [0, 1]$, where $(x, 1) \in \del \tilde X$, into itself.
 We suppose that 
 the support of 
 $h$ is inside the $\delta/2$-nbhd of $\bigcup_{(x,1) \in \del \tilde X \cap E} \{x\} \x [0, 1]$.
 This $h$ satisfies (3) of Theorem~\ref{shrinkinglocallycompact} so
$\pi$ is a near-homeomorphism which yields the claimed homeomorphism $\beta$.
\end{proof}

\bigskip{\Large{\section{Shrinkable decompositions}}}
\bigskip\large

The following notions are often used
to describe types of  decompositions which 
turn out to be shrinkable.

\begin{defn}
Let $\mathcal D$ be a usc decomposition of  $\R^n$.
\begin{itemize}
\item
$\mathcal D$ is  \emph{cell-like} if every decomposition element is cell-like,
\item
$\mathcal D$ is  \emph{cellular} if every decomposition element is cellular,
\item
the decomposition elements are \emph{flat arcs} if 
for every $D \in \mathcal D$ there is a homeomorphism $h \co \R^n \to \R^n$ such that 
$h(D)$ is a straight line segment,
\item
$\mathcal D$ is  \emph{starlike} if 
every decomposition element $D$
is a starlike set, that is, $D$ is a union of compact straight line segments 
with a common endpoint  $x_0 \in \R^n$,
\item
$\mathcal D$ is  \emph{starlike-equivalent} if 
for every $D \in \mathcal D$ there is a homeomorphism $h \co \R^n \to \R^n$ such that 
$h(D)$ is starlike,
\item
$\mathcal D$ is \emph{thin} 
if for every $D \in \mathcal D$
and every nbhd $U$ of $D$ there is an $n$-dimensional ball $B \subset  \R^n$ such that 
$D \subset B \subset U$ and $\del B$ is disjoint from the non-degenerate elements of $\mathcal D$,
\item
$\mathcal D$ is \emph{locally shrinkable} if
for each $D \in \mathcal D$ we have that for every
nbhd $U$ of $D$ and open cover $\mathcal V$ of $\R^n$
there is a
homeomorphism $h \co \R^n \to \R^n$ with support $U$ such that
$h(D) \subset V$ for some $V \in \mathcal V$,
\item
$\mathcal D$ 
\emph{inessentially spans}
the disjoint closed subsets $A,  B \subset \R^n$
if for every $\mathcal D$-saturated open cover $\mathcal U$ of $\R^n$
there is a homeomorphism $h \co \R^n \to \R^n$ which is 
$\mathcal U$-close to the identity and no element of $\mathcal D$ meets both of $h(A)$ and $h(B)$,
\item
the decomposition element $D$ has 
\emph{embedding dimension} $k$ if
for every $(n-k-1)$-dimensional
smooth submanifold $M$ of $\R^n$
and open cover $\mathcal V$ of $\R^n$
there is a homeomorphism $h \co \R^n \to \R^n$ which is 
$\mathcal V$-close to the identity,
$h (M) \cap D = \emptyset$ and
this is not true for $(n-k)$-dimensional submanifolds.
\end{itemize}
\end{defn}

Most of these notions have the corresponding verisons in arbitrary manifolds or spaces.
A condition that is obviously satisfied by at least $5$-dimensional Euclidean spaces  is the following.

\begin{defn}[Disjoint disks property]
The metric space $X$ 
has the \emph{disjoint disks property} if
for arbitrary maps $f_1$ and $f_2$ from $D^2$ to $X$ and for every $\ep > 0$
there are approximating maps $g_i$ from $D^2$ to $X$ 
$\ep$-close to $f_{i}$, $i = 1, 2$,
such that $g_1(D^2)$ and $g_2(D^2)$ are disjoint.
\end{defn}

The next theorem \cite{Ed78} is one of the fundamental results of decomposition theory, we omit its proof here.
\begin{thm}
Let $X$ be an at least $5$-dimensional manifold and let $\mathcal D$ be a cell-like decomposition of $X$.
Then  $\mathcal D$ is shrinkable if and only if $X_{\mathcal D}$ is finite dimensional and has the disjoint disks
property.
\end{thm}
Recall that a separable metric space is finite dimensional if every point has arbitrarily small nbhds having one less dimensional frontiers  
and
dimension $-1$ is by definition the dimension of the empty set.
For example, a manifold is finite dimensional.

In the following statement we enumerate
several conditions which imply that a (usc) decomposition is shrinkable.

\begin{thm}\label{shrinking_thm}
The following decompositions are strongly shrinkable:
\begin{enumerate}[\rm (1)]
\item
cell-like usc decompositions of a $2$-dimensional manifold,
\item
countable usc decompositions of $\R^n$ if the decomposition elements are flat arcs,
\item
countable and starlike usc decompositions of $\R^n$,
\item
countable and starlike-equivalent usc decompositions of $\R^3$,
\item
null and starlike-equivalent usc decompositions of $\R^n$,
\item
thin usc decompositions of $3$-manifolds,
\item
countable and thin usc decompositions of $n$-dimensional manifolds,
\item
countable and locally shrinkable  usc decompositions of a complete metric space if 
$\cup \mathcal H_{\mathcal D}$ is $G_{\delta}$,
\item
monotone usc decompositions of $n$-dimensional manifolds
if $\mathcal D$ inessentially spans 
every pair of disjoint, bicollared  $(n-1)$-dimensional spheres,
\item
null and cell-like decompositions of smooth $n$-dimensional manifolds 
if the embedding dimension of every $D \in \mathcal D$ is $\leq n-3$,
\end{enumerate}
\end{thm}

Before proving Theorem~\ref{shrinking_thm} let us make some observations and preparations.
At first,
 note that there are usc decompositions of $\R^3$
into straight line segments  which are not shrinkable: 
in the proof of Proposition~\ref{line_decomp} for any given compact metric space $Y$
we constructed a decomposition of $\R^3$ into straight line segments and singletons
such that $Y$ is a subspace of
the decomposition space.
 Since $\R^3$ is a complete metric space and the decomposition 
 is usc it is also shrinkable if and only if $\pi$ is approximable by homeomorphisms. 
 This means that if $Y$ cannot be embedded into $\R^3$, then 
 the decomposition space cannot be homeomorphic to $\R^3$
 and  then this decomposition is not shrinkable.

If the decomposition is countable, then 
we can shrink successively the decomposition elements 
if there is a guaranty of not expanding an already shrunken element 
while shrinking another one.
The next proposition is a technical tool for this process.

\begin{prop}\label{shrinking_lemma}
Let $\mathcal D$ be a countable usc decomposition 
of a locally compact metric space $X$.
Suppose for every $D \in \mathcal D$, for every $\ep > 0$ 
and for every homeomorphism $f \co X \to X$
there exists
a homeomorphism $h \co X \to X$ such that
\begin{enumerate}[\rm (1)]
\item
outside of the $\ep$-nbhd of $D$ the homeomorphism 
$h$ is the same  as $f$,
\item
$\mathrm {diam} \thinspace h(D) < \ep$ and
\item
for every $D' \in \mathcal D$
we have $\mathrm {diam} \thinspace h(D') < \ep + \mathrm {diam} \thinspace  f ( D')$. 
\end{enumerate}
Then $\mathcal D$ is strongly shrinkable.
\end{prop}
\begin{proof}[Sketch of the proof]
Let $\ep > 0$ and let $\mathcal U$ be a $\mathcal D$-saturated open cover of $X$. 
We enumerate the non-degenerate elements of $\mathcal D$ which have diameter at least $\ep/2$ as $D_1, D_2, \ldots$.
We can find $\mathcal D$-saturated open sets $U_1, U_2, \ldots$ such that 
for all $n$ we have $D_n \subset U_n$ and all sets $U_n$ are pairwise disjoint or coincide. These $U_n$ are subsets 
of sets in $\mathcal U$ and they will ensure $\mathcal U$-closeness.
We produce a sequence ${\mathrm {id}}=h_0, h_1, \ldots$ of self-homeomorphisms of $X$
and a sequence $C_1, C_2, \ldots$ of $\mathcal D$-saturated closed  nbhds of $D_1, D_2, \ldots$, respectively, 
such that a couple of conditions are satisfied for every $n \geq 1$:
\begin{enumerate}[(a)]
\item
$h_n |_{X - U_n} = h_{n-1}|_{X - U_n}$,
\item
${\mathrm {diam}}\thinspace h_n ( D_n) < \ep$,
\item
for every $D \in \mathcal D$ we have ${\mathrm {diam}}\thinspace h_n ( D) < (1 - \frac{1}{2^n})\frac{\ep}{2} 
 + {\mathrm {diam}} \thinspace D$,
\item
$h_{n+1}|_{C_1 \cup \cdots \cup C_n} = h_{n}|_{C_1 \cup \cdots \cup C_n}$,
\item
if some $D \in \mathcal D$  is in  $C_n$, then  ${\mathrm {diam}}\thinspace h_n ( D) < \ep$ and
\item
$h_n = h_{n-1}$ if ${\mathrm {diam}}\thinspace h_{n-1} ( D_n) < \ep$.
\end{enumerate}
The sets $C_n$ serve as protective buffers in which no further motion will occur. 
For $n=1$ by the conditions (1), (2) and (3) in the statement of Proposition~\ref{shrinking_lemma} 
with the choice $f={\mathrm {id}}$ we can find a homeomorphism $h_1 \co X \to X$ 
satisfying (a), (b) and (c) and also an appropriate $C_1$ such that (d) and (e) are satisfied as well.
If $h_k$ and $C_k$ are defined already for $1 \leq k \leq n$, then 
we find $h_{n+1}$ and $C_{n+1}$ as follows.
If ${\mathrm {diam}}\thinspace h_{n} ( D_{n+1}) < \ep$,
then let $h_{n+1} = h_n$. If the diameter of $h_{n} ( D_{n+1} )$
is at least $\ep$, then 
by the conditions (1), (2) and (3) with the choice $f=h_n$
we can find a homeomorphism $h_{n+1} \co X \to X$ satisfying
\begin{enumerate}[(i)]
\item
$h_{n+1}|_{X - U_{n+1}} = h_{n}|_{X - U_{n+1}}$
\item
${\mathrm {diam}}\thinspace h_{n+1} ( D_{n+1}) < \ep/ 2^{n+2}$,
\item
for every $D \in \mathcal D$ we have ${\mathrm {diam}}\thinspace h_{n+1} ( D) < \ep/{2^{n+2}}
 + {\mathrm {diam}} \thinspace h_n(D)$
\end{enumerate}
furthermore (iii) and (c) imply that 
for every $D \in \mathcal D$ we have
$$
{\mathrm {diam}}\thinspace h_{n+1} ( D) <\ep/{2^{n+2}} + 
\left(1 - \frac{1}{2^n}\right)\frac{\ep}{2} 
 + {\mathrm {diam}} \thinspace D = 
 \left(1 - \frac{1}{2^{n+1}}\right)\frac{\ep}{2} 
 + {\mathrm {diam}} \thinspace D 
$$
so (a), (b) and (c) are satisfied. It is not too difficult to get (d) and (e) with some $C_{n+1}$ as well.
 After having all $h_1, h_2, \ldots$ and $C_1, C_2, \ldots$ with properties (a)-(f)
 it is easy to see by (d), (e) and (f) that every $D \in \mathcal D$ which is in $C_1 \cup \cdots \cup C_n$ is shrunk
by $h_n$ to size smaller than $\ep$ and other $h_{n+i}$ does not modify this.
If some $D \in \mathcal D$ 
had  diameter smaller than $\ep/2$ originally, then  (c) implies that its diameter is smaller than  $\ep$ during all the process.
These imply that the sequence $h_1, h_2, \ldots$ is locally stationary and it converges to a shrinking homeomorphism $h$.
\end{proof}

We are going to give a sketch of the proof of Theorem~\ref{shrinking_thm}. 
For the detailed proof of (1) see \cite{Mo25}, for the proofs of
(2) and (3) see \cite{Bi57}, for the proof of (4) see \cite{DS83}, for (5) see \cite{Be67}, for (6) see \cite{Wo77} and 
for (7), (8), (9) and (10) see \cite{Pr66}, \cite{Bi57}, \cite{Ca78} and \cite{Ca79, Ed16}, respectively.
\begin{proof}[Sketch of the proof of Theorem~\ref{shrinking_thm}]
(1) follows from the fact that in a $2$-dimensional manifold $X$
a cell-like decomposition is thin. The reason of this is that an arbitrarily  small $2$-dimensional disk nbhd $B$  with 
the property 
$\del B \cap (\cup \mathcal H_{\mathcal D}) = \emptyset$
can be obtained by finding the circle $\del B$  in $X$ as a limit of a sequence of maps $f_n \co S^1 \to X$
avoiding smaller and smaller decomposition elements.
A thin usc decomposition of a $2$-dimensional manifold is shrinkable if
the points of $\pi(\bigcup \mathcal H_{\mathcal D})$ do not converge to each other in a too
complicated way.
Since the quotient space $X_{\mathcal D}$
can be filtered in  a way which implies this, 
the decomposition map $\pi$ can be successively approximated by maps which are homeomorphisms on the induced 
filtration in $X$.

 (2)-(5) follows from Proposition~\ref{shrinking_lemma}: the flat arcs, starlike sets and starlike-equivalent sets
can be shrunk successively because of geometric reasons.

 To prove (6) and (7) we also use Proposition~\ref{shrinking_lemma}.
 Let $D \in \mathcal D$ be a non-degenerate decomposition element, $U$ a nbhd of $D$
 and let $B$ be a ball such that
 $D \subset B \subset U$ and $\del B$ is disjoint from the non-degenerate elements of $\mathcal D$.
 After applying a self-homeomorphism of $X$, we can suppose that 
 $B$ is the unit ball. 
 Let $k$ be some large enough integer and let $1 > \de_0 > \de_1 > \cdots > \de_{k-1} > 0$ be such that 
 if  $D' \in \mathcal D$ intersects 
 the $\de_{n+1}$-nbhd of $\del B$, then $D'$ is inside the $\de_{n}$-nbhd of $\del B$.
 Define a homeomorphism $f \co B \to B$
 which is the identity on $\del B$, keeps the center of $B$ fixed and
 on each radius $R$ the point at distance $\de_n$ from $\del B$, where $1 \leq n \leq k-1$, is mapped to 
 the point at distance $n/k$ from the center. We require that the homeomorphism $f$ is linear between these points. 
  After applying this homeomorphism, 
 every $D' \in \mathcal D$ in $B$ is shrunk to size small enough.
  
 In the proof of (8) 
we enumerate the non-degenerate decomposition 
elements
and we construct a sequence of homeomorphisms of the ambient space
which shrink the decomposition elements 
successively using the locally shrinkable property.

 To prove (9) for a given $\ep > 0$
we cover the manifold by two collections $\{ B_{\alpha} \}_{\alpha \in A}$ and $\{ B_{\alpha}' \}_{\alpha \in A}$ of 
$n$-dimensional balls such that $B_{\alpha} \subset \mathrm {int} \thinspace B_{\alpha}'$ and $\mathrm {diam} \thinspace B_{\alpha}' < \ep$.
 Then the closed sets $\pi (\del B_{\alpha})$ and $\pi ( \del B_{\alpha}')$
are made disjoint by applying homeomorphisms $h_{\alpha}$ successively.
This implies that the homeomorphism $h$ obtained by composing all  the homeomorphisms $h_{\alpha}$
is such that for every $D \in \mathcal D$ the set $h(D)$ is fully contained in some 
ball $B_{\alpha}'$ so its diameter is smaller than $\ep$.

 In the proof of (10) at first we obtain that every decomposition element $D$ is cellular because of the following.
 By assumption $D$ is cell-like and 
it behaves like an at most $(n-3)$-dimensional submanifold so 
the $2$-skeleton of the ambient manifold is disjoint from $D$.
This means that $D$ satisfies the cellularity criterion since 
the $2$-skeleton carries the fundamental group.
Hence $D$ is cellular, which implies that it is contained in an $n$-dimesional ball and also in a starlike-equivalent
set $C$ of embedding dimension $\leq n-2$.
Now it is possible to use an argument similar to the proof of Proposition~\ref{shrinking_lemma}: 
we can shrink $C$ to become smaller than an $\ep > 0$
by successively compressing $C$  and in each iteration carefully controlling and avoiding other decomposition elements close to $C$ 
which would become too large during the compression procedure.
\end{proof}

%
%


\end{document}